\documentclass[11pt]{article}

\usepackage{amsmath,amsthm,enumitem,amssymb,graphicx}
\usepackage[margin=2cm]{caption}
\usepackage{lipsum}
\usepackage{hyperref}
\usepackage[title]{appendix}

\usepackage{geometry}
\newcommand{\mr}{\mathring}
\newcommand{\wt}{\widetilde}
\newcommand{\wh}{\widehat}
\newcommand{\Z}{\mathbb{Z}}

\newcommand{\R}{\mathbb{R}}
\newcommand{\F}{\mathcal{F}}
\newcommand{\C}{\mathcal{C}}

\newcommand{\M}{\mathcal{M}}
\newcommand{\OO}{\mathcal{O}}

\newcommand{\tri}{\triangle}

\DeclareMathOperator{\id}{id}

\DeclareMathOperator{\closure}{cl}
\DeclareMathOperator{\cone}{cone}
\DeclareMathOperator{\intr}{int}

\DeclareMathOperator{\flow}{flowspace}
\DeclareMathOperator{\dimension}{dimension}
\DeclareMathOperator{\St}{St}
\DeclareMathOperator{\LD}{LD}

\renewcommand{\phi}{\varphi}

\numberwithin{equation}{section}
\newtheorem{theorem}[equation]{Theorem}
\newtheorem{proposition}[equation]{Proposition}
\newtheorem{lemma}[equation]{Lemma}
\newtheorem{corollary}[equation]{Corollary}

\newtheorem*{TST}{Transverse Surface Theorem}
\newtheorem*{appthm}{Theorem \ref{generalized Fried}}
\newtheorem*{LRthm}{Theorem \ref{stable loop theorem}}
\newtheorem*{generalTST}{Theorem \ref{myTST}}
\newtheorem*{linkcor}{Corollary \ref{link corollary}}

\theoremstyle{definition}
\newtheorem{definition}[equation]{Definition}

\newtheorem{observation}[equation]{Observation}
\newtheorem*{question*}{Question}

\theoremstyle{remark}
\newtheorem{remark}[equation]{Remark}

\begin{document}

\title{Stable loops and almost transverse surfaces}
\author{Michael Landry}
%\address{Yale University}
%\email{michael.landry@yale.edu}

\date{}
\maketitle

\begin{figure}[h]
\centering
\includegraphics[height=2.2in]{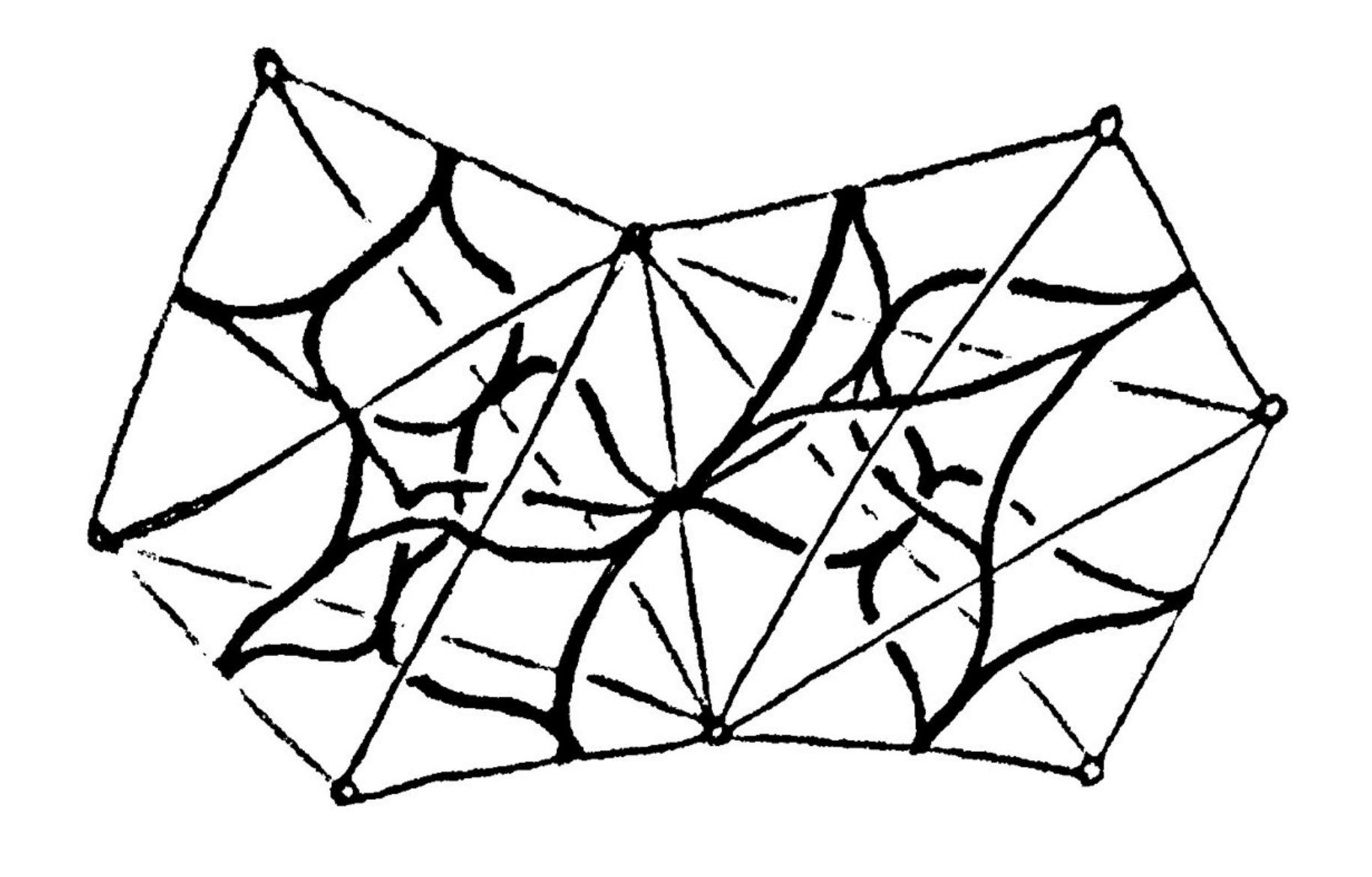}
\caption{A portion of a veering triangulation and its stable train track.}
\label{visualizingthetraintrack}
\end{figure}

\begin{abstract}
We show that the cone over a fibered face of a compact fibered hyperbolic 3-manifold is dual to the cone generated by the homology classes of finitely many curves called \emph{minimal stable loops} living in the associated veering triangulation. We also present a new, more hands-on proof of Mosher's Transverse Surface Theorem. 
\end{abstract}

%\tableofcontents

\section{Introduction}
In this paper we use veering triangulations to study the suspension flows of pseudo-Anosov homeomorphisms on compact surfaces, which we call \textbf{circular pseudo-Anosov flows}. 

Let $M$ be a compact hyperbolic 3-manifold with fibered face $\sigma\subset H_2(M,\partial M;\R)$. There is a circular pseudo-Anosov flow $\phi$, unique up to reparameterization and conjugation by homeomorphisms isotopic to the identity, which organizes the monodromies of all fibrations of $M$ corresponding to $\sigma$ \cite{Fri79}. We call this flow the \textbf{suspension flow} of the fibered face.
The suspension flow $\phi$ has the following property: a cohomology class $u\in H^1(M;\R)$ is Lefschetz dual to a class in $\cone(\sigma):=\R_{\ge0}\cdot\sigma$ if and only if $u$ is nonnegative on $\C_\phi$, the \textbf{cone of homology directions} of $\phi$. Hence computing $\C_\phi$ is equivalent to computing $\cone(\sigma)$.

For a flow $F$ on $M$, $\C_F$ is the smallest closed cone containing the projective accumulation points of homology classes nearly closed orbits of $F$ \cite{Fri79, Fri82}. For our circular pseudo-Anosov flow $\phi$, $C_\phi$ has a more convenient characterization as the smallest closed cone containing the homology classes of the closed orbits of $\phi$. Our main result gives a characterization of $\C_\phi$ in terms of the veering triangulation $\tau$ associated to $\sigma$. 

\begin{LRthm}[Stable loops]
Let $M$ be a compact hyperbolic 3-manifold with fibered face $\sigma$. Let $\tau$ and $\phi$ be the associated veering triangulation and circular pseudo-Anosov flow, respectively. Then $\C_\phi$ is the smallest convex cone containing the homology classes of the minimal stable loops of $\tau$.
\end{LRthm}

\begin{figure}[h]
\centering
\includegraphics[height=2in]{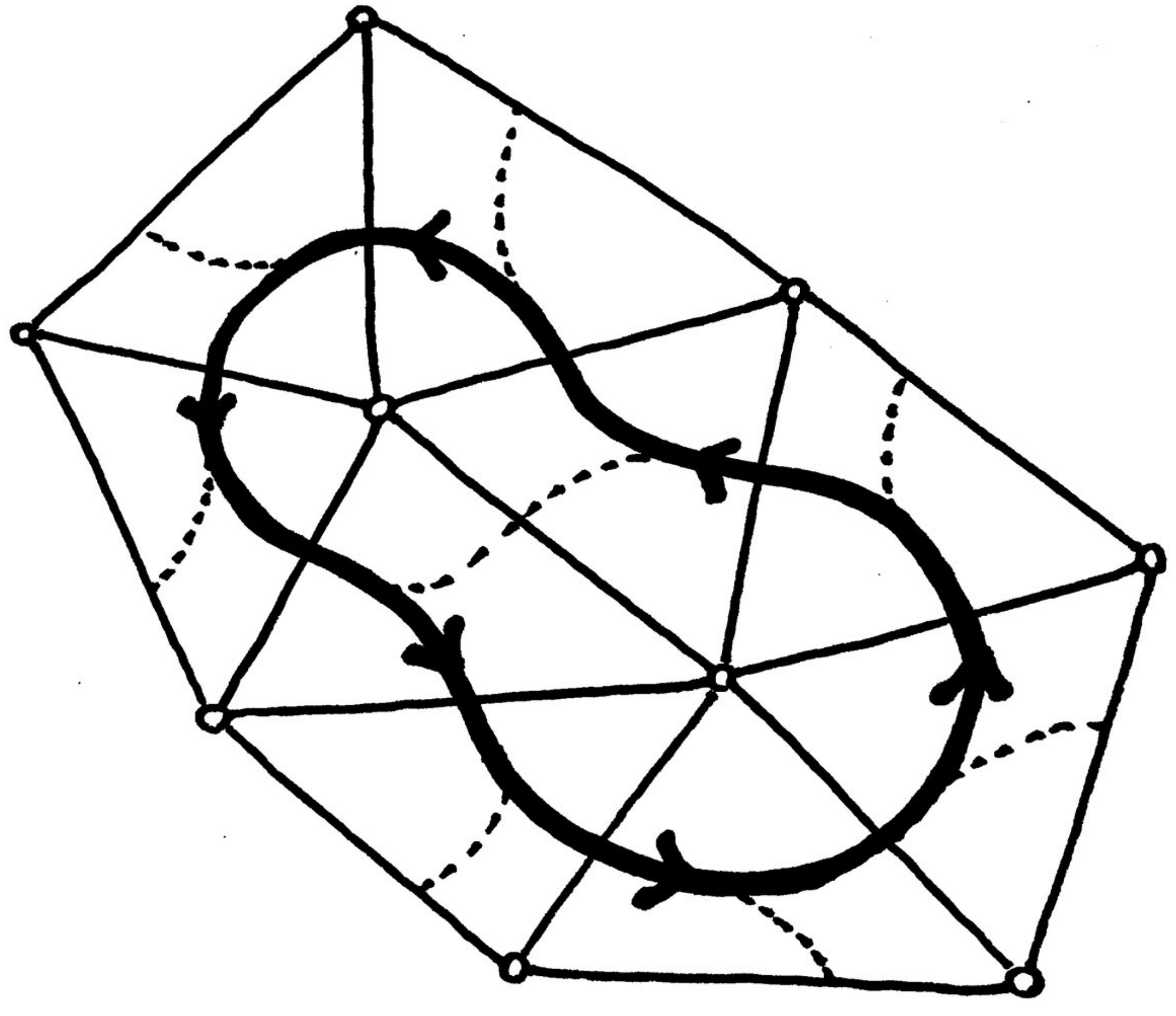}
\caption{A stable loop.}
\label{lazyriver}
\end{figure}

The veering triangulation $\tau$ is a taut ideal triangulation of a cusped hyperbolic 3-manifold $M'$ obtained from $M$ by deleting finitely many closed curves from $\intr(M)$. In Theorem \ref{stable loop theorem} we are viewing it as embedded in $M$ as an ideal triangulation of an open submanifold.
The \textbf{stable loops} of $\tau$ are a family of closed curves carried by the so-called \textbf{stable train track of $\tau$}, which lies in the 2-skeleton of $\tau$ and is defined in Section \ref{sec:traintrack}.  They correspond to nontrivial elements of the fundamental groups of leaves of the stable foliation of $\phi$. A \textbf{minimal stable loop} is a stable loop traversing each switch of the stable train track at most once.
Our proof uses the fact that the 2-skeleton of the veering triangulation is a branched surface, and (Proposition \ref{flipping fibers}) a surface carried by this branched surface is a fiber of $M'$ if and only if it is \textbf{infinitely flippable} in a sense defined in Section \ref{sec:flippability}. 

Flippability is a condition depending on the combinatorics of the veering triangulation. We begin the paper by developing some of these combinatorics. This allows us to give a hands-on, combinatorial proof of Mosher's Transverse Surface Theorem\footnote{This is actually a slight generalization of Mosher's theorem, as he proved Theorem \ref{myTST} only in the case when $\partial M=\varnothing$. He later proved a version for pseudo-Anosov flows which are not necessarily circular \cite{Mos92b}.} \cite{Mos91}:

\begin{generalTST}[Almost transverse surfaces]
Let $M$ be a compact hyperbolic 3-manifold, with a fibered face $\sigma$ of $B_x(M)$ and associated suspension flow $\phi$. Let $\alpha\in H_2(M,\partial M)$ be an integral homology class. Then $\alpha\in \cone(\sigma)$ if and only if $\alpha$ is represented by a surface almost transverse to $\phi$.
\end{generalTST}

A surface is \textbf{almost transverse} to $\phi$ if it is transverse to a closely related flow $\phi^\sharp$, obtained from $\phi$ by a process called \textbf{dynamically blowing up} singular orbits.
The precise definition of almost transversality is found in Section \ref{TSTstatement}. 

Loosely speaking, the strategy of our proof of Theorem \ref{myTST} is to arrange a surface to lie in a regular neighborhood of the veering triangulation away from the singular orbits of $\phi$, and then use our knowledge of  how the veering triangulation sits in relation to $\phi$ in order to appropriately blow up the flow near the singular orbits. Mosher, without the machinery of veering triangulations available to him, performed a deep analysis of the dynamics of the lift of $\phi$ to the cyclic cover of $M$ associated to the surface. Including this dynamical analysis, his complete proof of the theorem spans \cite{Mos89,Mos90,Mos91}.

Taut ideal triangulations were introduced by Lackenby in \cite{Lac00} as combinatorial analogues of taut foliations, where he uses them to give an alternative proof of Gabai's theorem that the singular genus of a knot is equal to its genus. He states ``One of the principal limitations of taut ideal triangulations is that they do not occur in closed 3–manifolds," and asks:

\begin{question*}[Lackenby]
Is there a version of taut ideal triangulations for closed 3–manifolds?
\end{question*}

While we do not claim a comprehensive answer to Lackenby's question, a a theme of this paper and \cite{Lan18} is that for fibered hyperbolic 3-manifolds (possibly closed), a veering triangulation of a dense open subspace is a useful version of a taut ideal triangulation.

Veering triangulations are introduced by Agol in \cite{Ago10}. There  is a canonical veering triangulation associated to any fibered face of a hyperbolic 3-manifold, and Gu\'eritaud showed \cite{Gue15} that it can be built directly from the suspension flow. If taut ideal triangulations are combinatorial analogues of taut foliations, Gu\'eritaud's construction allows us to view veering triangulations as combinatorializations of pseudo-Anosov flows\footnote{The forthcoming paper [SchSeg19] will make this explicit.}.

We include two appendices. In Appendix \ref{Fried appendix} we prove that a result of Fried from \cite{Fri79}, which was stated and proved for closed hyperbolic 3-manifolds, holds for manifolds with boundary. The result, which we have been assuming so far in this introduction and which is necessary for the results of this paper, states that the suspension flow $\phi$ is canonically associated to $\sigma$ in the sense described above. In Appendix \ref{TBSexpanded}, we explain how to use Theorem \ref{myTST} to show the results of \cite{Lan18} hold for manifolds with boundary. In particular we get the following corollary.

\begin{linkcor}
Let $L$ be a fibered hyperbolic link with at most 3 components. Let $M_L$ be the exterior of $L$ in $S^3$. Any fibered face of $B_x(M_L)$ is spanned by a taut branched surface.
\end{linkcor}

\subsection{Acknowledgements}
I thank Yair Minsky, James Farre, Samuel Taylor, and Ian Agol for stimulating conversations. I additionally thank Yair Minsky, my PhD advisor, for generosity with his time and attention during this research and throughout my time as a graduate student.

I gratefully acknowledge the support of the National Science Foundation Graduate Research Fellowship Program under Grant No. DGE-1122492, and of the National Science Foundation under Grant No. DMS-1610827 (PI Yair Minsky). Any opinions, findings, and conclusions or recommendations expressed in this material are my own and do not necessarily reflect the views of the National Science Foundation.

\section{Preliminaries}

In this paper, all manifolds are orientable and all homology and cohomology groups have coefficients in $\R$.

\subsection{The Thurston norm, fibered faces, relative Euler class}\label{norm background}

We review some facts about the Thurston norm, which can be found in \cite{Thu86}. Let $M$ be a compact, irreducible, boundary irreducible, atoroidal, anannular 3-manifold ($\partial M$ may be empty). If $S$ is a connected surface embedded in $M$,  define 
\[
\chi_-(S)=\max\{0,-\chi(S)\}
\]
where $\chi$ denotes Euler characteristic. If $S$ is disconnected, let $\chi_-(S)=\sum_i \chi_-(S_i)$ where the sum is taken over the connected components of $S$. For any integral homology class $\alpha\in H_2(M,\partial M)$, we can find an embedded surface representing $\alpha$. Define
\[
x(\alpha)=\min\{\chi_-(S)\mid \text{$S$ is an embedded surface representing $\alpha$}\}.
\]
Then $x$ extends by linearity and continuity from the integer lattice to a vector space norm on $H_2(M,\partial M)$ called the \textbf{Thurston norm} \cite{Thu86}. We mention that Thurston defined $x$ more generally to be a seminorm on $H_2(M,\partial M)$ for any compact orientable $M$. However, in this paper $x$ will always be a norm, since the manifolds we consider will not admit essential surfaces of nonnegative Euler characteristic.

The unit ball of $x$ is denoted by $B_x(M)$. As a consequence of $x$ taking integer values on the integer lattice, $B_x(M)$ is a finite-sided polyhedron with rational vertices. Our convention in this paper is that a \textbf{face} of $B_x(M)$ is a \emph{closed} cell of the polyhedron.

We say an embedded surface $S$ is 
%\textbf{taut}\footnote{Some authors have used ``norm-minimizing" but I prefer ``taut," because to say $S$ is norm-minimizing suggests it merely realizes the minimal $\chi_-$ in $[S]$. An example of such a surface which is not taut is the boundary of a solid torus in $M$, which realizes the minimal $\chi_-$ in the homology class 0.} 
if it is incompressible and realizes the minimal $\chi_-$ in $[S]$. 
If $\Sigma\subset M$ is the fiber of a fibration $\Sigma\hookrightarrow M\to S^1$ then $\Sigma$ is taut, any taut surface representing $[\Sigma$]  is isotopic to $\Sigma$, and $[\Sigma]$ lies in $\intr(\cone(\sigma))$ for some top-dimensional face $\sigma$ of $B_x(M)$. Moreover, any other integral class representing a class in $\intr(\cone(\sigma))$ is represented by the fiber of some fibration of $M$ over $S^1$. Such a top-dimensional face $\sigma$ is called a \textbf{fibered face}. 

Let $\xi$ be an oriented plane field on $M$ which is transverse to $\partial M$. If we fix an outward pointing section of $\xi|_{\partial M}$, this determines a \textbf{relative Euler class} $e_\xi\in H^2(M,\partial M)$. For a relative 2-cycle $S$, $e_\xi([S])$ is the first obstruction to finding a nonvanishing section of $\xi|_S$ agreeing with the outward pointing section on $\partial S\subset\partial M$. A reference on relative Euler class is \cite{Sha73}. 

If $\phi$ is a flow on $M$ tangent to $\partial M$, Let $T\phi$ be the oriented line field determined by the tangent vectors to orbits of $\phi$. Let $\xi_\phi$ be the oriented plane field which is the quotient bundle of $TM$ by $T\phi$. We can think of $\xi_\phi$ as a subbundle of $TM$ by choosing a Riemannian metric and identifying $\xi_\phi$ with the orthogonal complement of $T\phi$. For notational simplicity we define $e_{\phi}=e_{\xi_\phi}$, the relative Euler class of $\xi_\phi$.

Fix a fibration $Y\hookrightarrow M\to S^1$, which allows us to express $M\cong (Y\times[0,1])/(y,1)\sim (g(y),0)$ for some homeomorphism $g$ of $Y$. Let $TY$ be the tangent plane field to the foliation of $M$ by $(Y\times\{t\})$'s. Let $\phi$ be the \textbf{suspension flow} of $g$, which moves points in $M$ along lines $(y,t)$ for fixed $y$, gluing by $g$ at the boundary of $Y\times[0,1]$. We have $\xi_\phi\cong TY$ and hence $e_{\phi}=e_{TY}$.
For some fibered face $\sigma$, we have $[Y]\in \intr(\cone(\sigma))$.  We have $x([Y])=-\chi(Y)=-e_{TY}([Y])$, i.e. $x$ and $e_{TY}$ agree on $[Y]$. In fact, more is true: $\cone(\sigma)$ is exactly the subset of $H_2(M,\partial M)$ on which $-e_{TY}$ and $x$ agree.

It can be fruitful to think of a properly embedded surface $S$ in $M$ as representing both a homology class in $H_2(M,\partial M)$ and a cohomology class in $H^1(M)$ mapping homology classes of closed curves to their intersection number with $S$. As such we will sometimes think of $x$ as a norm on $H^1(M)$ via Lefschetz duality. The image in $H^1(M)$ of a face $\sigma$ of $B_x(M)$ will be denoted $\sigma_{\LD}$, and in general the subscript $LD$, when attached to an object, will denote the Lefschetz dual of that object.

\subsection{Circular pseudo-Anosov flows}

\begin{figure}[h]
\centering
\includegraphics[height=3in]{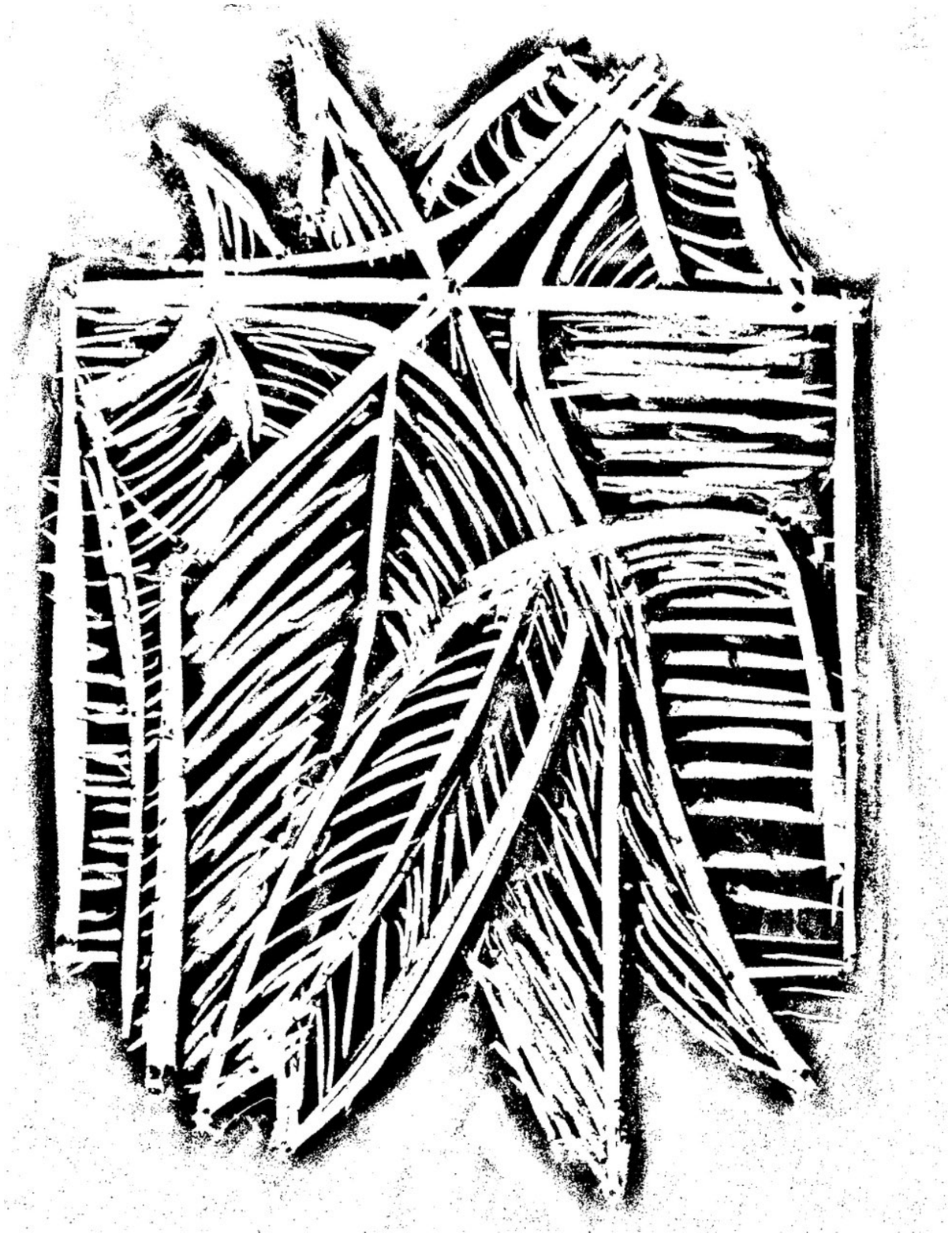}
\includegraphics[height=3in]{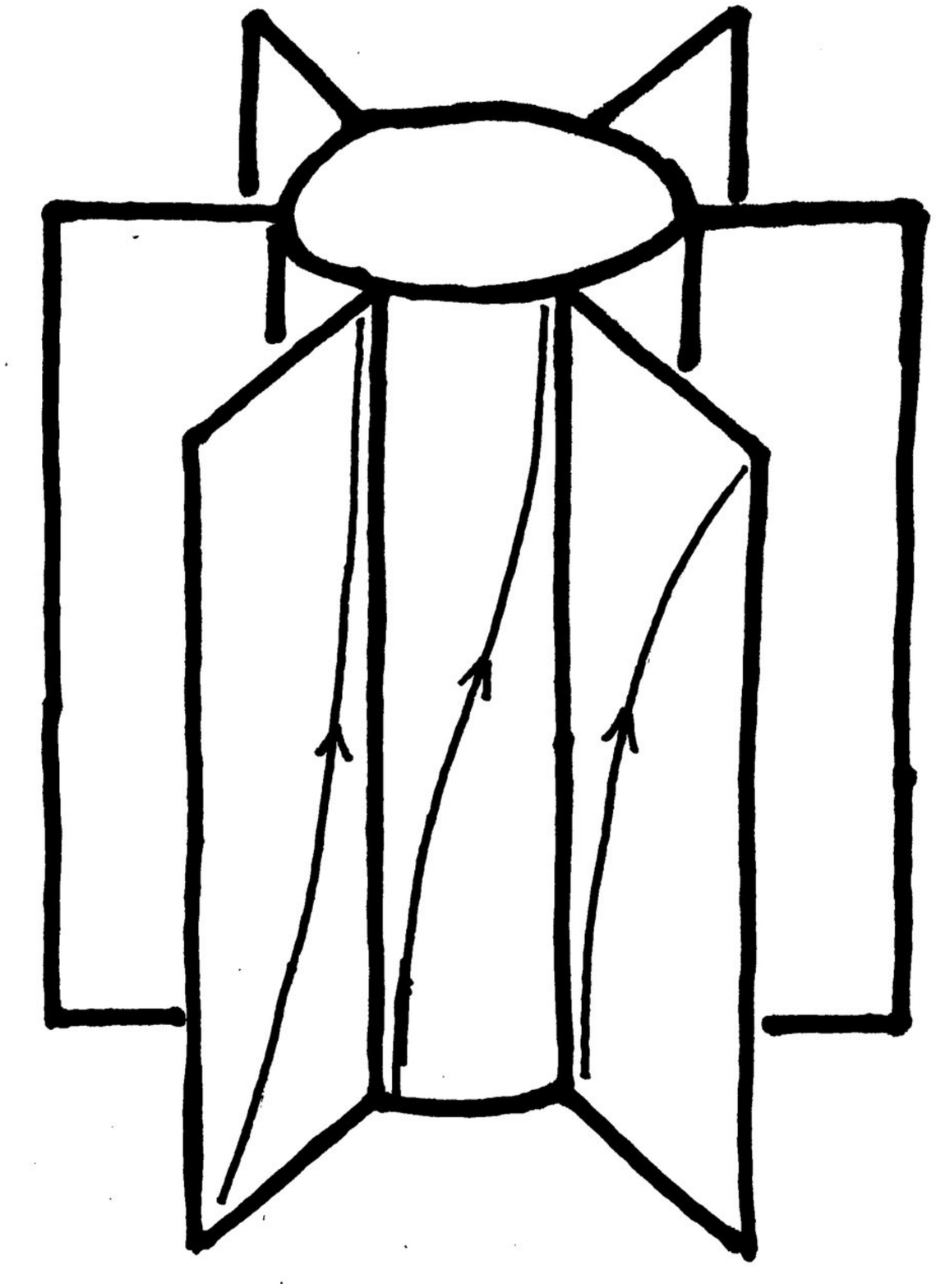}
\caption{On the left we see leaves of the transverse foliations of a pseudo-Anosov flow near a singular orbit. On the right we see the behavior of a circular pseudo-Anosov flow near a boundary component.}
\label{pAdef}
\end{figure}

Let $F$ be a flow on $M$ which is tangent to $\partial M$.
Recall that a \textbf{cross section} to $F$ is a fiber of a fibration $M\to S^1$ whose fibers are transverse to $F$. Flows which admit cross sections are called \textbf{circular}. We call the suspension flow $\phi$ of a pseudo-Anosov map $g\colon Y\to Y$ on a compact surface $Y$ a \textbf{circular pseudo-Anosov flow}.

Such a $g$ preserves two transverse measured foliations called the \textbf{stable} and \textbf{unstable} foliations of $g$ which are transverse everywhere except the boundary of the surface (see \cite{Thu88}). These foliations give two transverse codimension-1 foliations in $M$ preserved by $\phi$ called the \text{stable} and \textbf{unstable} foliations of $\phi$. The closed orbits corresponding to the singular points of $g$ lying in $\intr(Y)$ are called \textbf{singular orbits}. The closed orbits lying on $\partial M$, which correspond to the boundary components of certain stable and unstable leaves, are called \textbf{$\partial$-singular orbits}. See Figure \ref{pAdef}.

If $Z$ is a cross section to a flow $\phi$, the first return map of $Z$ is the map sending $z\in Z$ to the first point in its forward orbit under $\phi$ lying
in $Z$. If $\phi$ is a circular pseudo-Anosov flow, the first return map of any cross section to $\phi$ will be pseudo-Anosov (this was proved in \cite{Fri79} for closed manifolds, but the proof in general is essentially the same).

\subsection{Review of the veering triangulation}

\begin{figure}[h]
\centering
\includegraphics[height=2in]{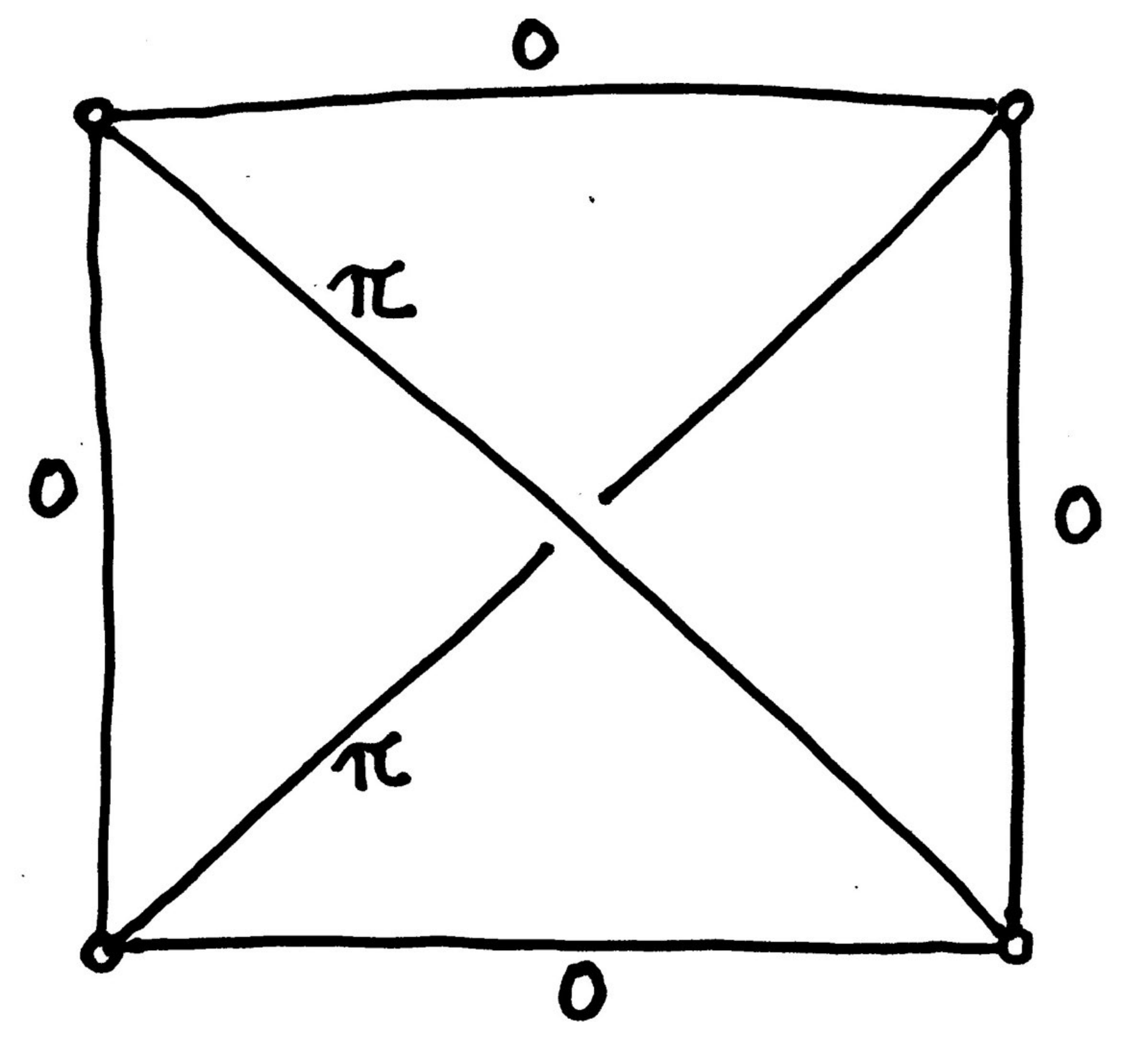}
\caption{A taut tetrahedron. The labels on edges indicate interior angles, and our convention is that the coorientation of each face points out of the page.}
\label{tauttetrahedron}
\end{figure}

A \textbf{taut tetrahedron} is an ideal tetrahedron with the following edge and face decorations. Four edges are labeled 0, and two are labeled $\pi$. Two faces are cooriented outwards and two are cooriented inwards, and faces of opposite coorientation meet only along edges labeled 0. The edge labels should be thought of as the interior angles of the corresponding edges. See Figure \ref{tauttetrahedron}. We define the \textbf{top} (resp. \textbf{bottom}) of a taut tetrahedron $t$ to be the union of the two faces whose coorientations point out of (resp. into) $t$. The \textbf{top} (resp. \textbf{bottom}) \textbf{$\pi$-edge} of $t$ will be the edge labeled $\pi$ lying in the top (resp. bottom) of $t$.

\begin{figure}[h]
\centering
\includegraphics[height=1.7in]{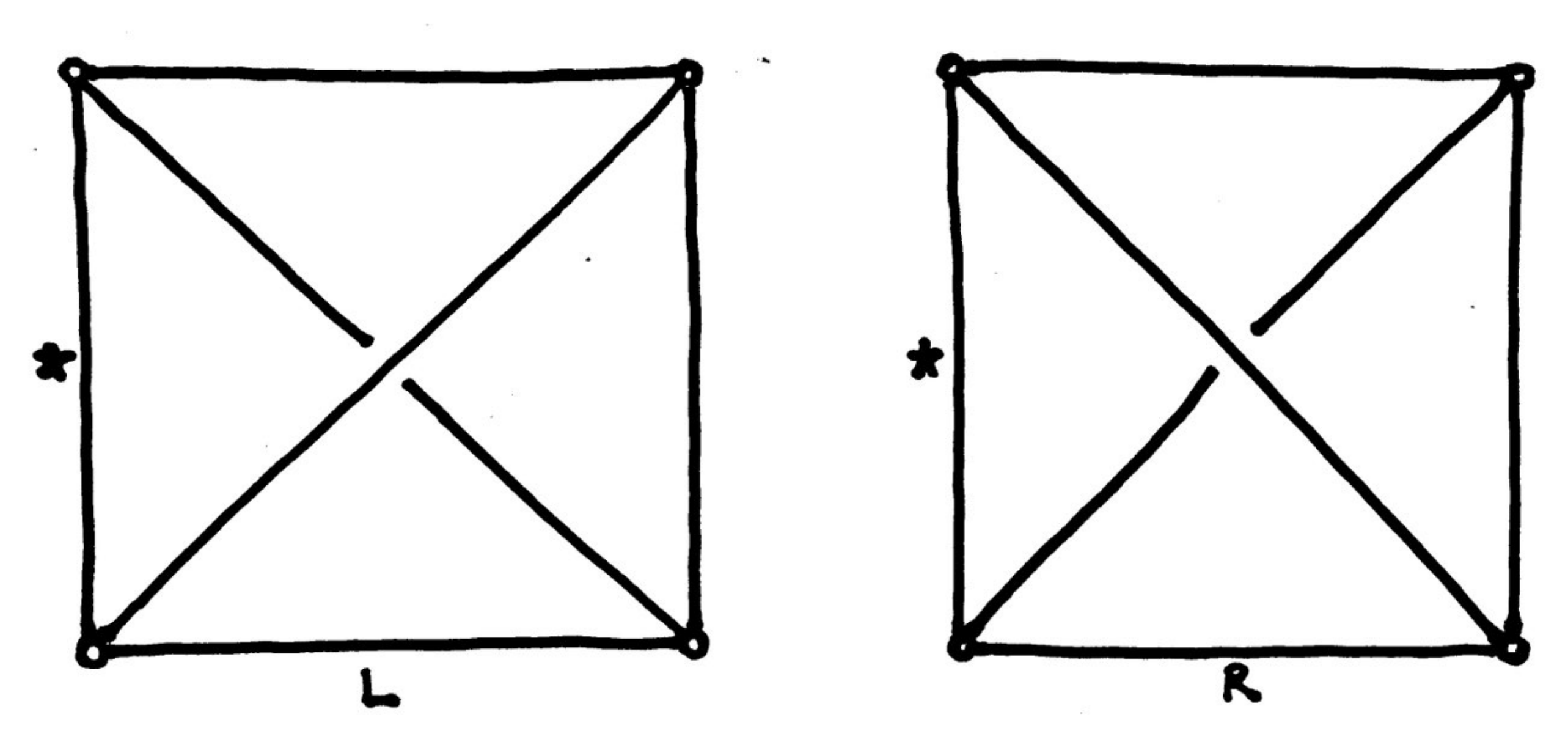}
\caption{The taut tetrahedra above are of types $L$ and $R$ when the starred edge is distinguished.}
\label{RandLtetrahedra}
\end{figure}

Up to orientation-preserving combinatorial equivalence, there are two types of taut tetrahedron with a distinguished 0-edge. We call these $L$ and $R$ and they are shown in Figure \ref{RandLtetrahedra}. 
\begin{definition}
A \textbf{taut ideal triangulation} of a 3-manifold is an ideal  triangulation by taut tetrahedra such that 
\begin{enumerate}
\item when two faces are identified their coorientations agree,
\item the sum of interior angles around a single edge is $2\pi$, and
\item no two interior angles of $\pi$ are adjacent around an edge.
\end{enumerate}
\end{definition}

A consequence of the third condition above is that each edge of a taut ideal triangulaton has degree $\ge 4$.

\begin{definition}
Let $e$ be an edge of a taut ideal triangulation $\triangle$. If $e$ has the property that all tetrahedra for which $e$ is a 0-edge are of type $R$ when $e$ is distinguished, we say $e$ is \textbf{right veering}. Symmetrically, if they are all of type $L$ we say $e$ is \textbf{left veering}. If every edge of $\triangle$ is either right or left veering, $\triangle$ is \textbf{veering}.
\end{definition}

\begin{figure}[h]
\centering
\includegraphics[height=2in]{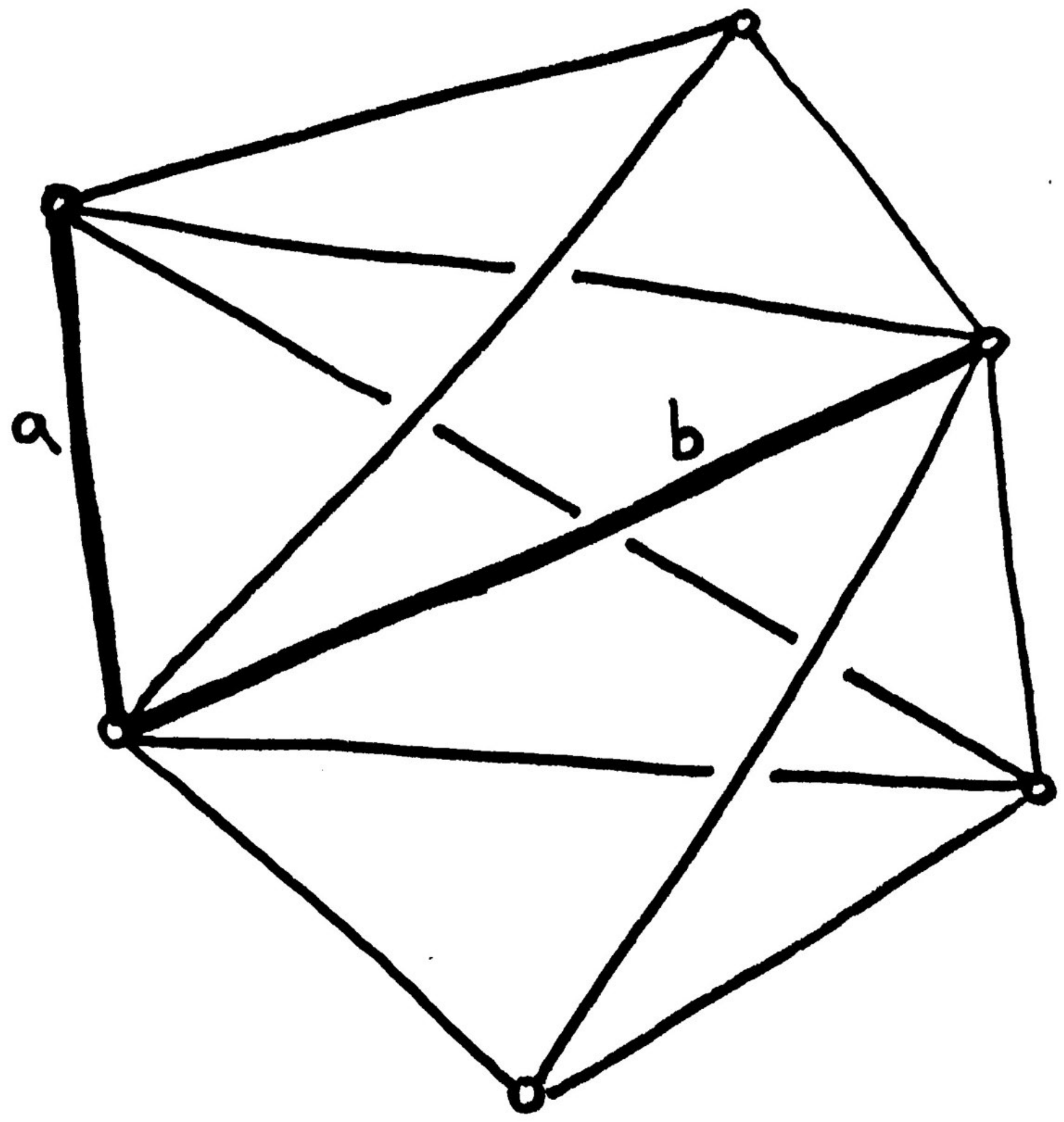}
\caption{Part of a veering triangulation. The edge $a$ is left veering, while $b$ is right veering.}
\label{veeringdeffigure}
\end{figure}

The 0 and $\pi$ labels tell us how to ``pinch" a taut ideal triangulation along its edges so that the 2-skeleton has a well-defined tangent space at every point. We will always assume the 2-skeleton of a taut ideal triangulation is embedded in such a way.

\subsection{The veering triangulation of a fibered face}\label{canonical veering section}

In this section we move somewhat delicately between compact manifolds and their interiors. The reason for this is that we wish to work in a compact manifold and study its second homology rel boundary, but veering triangulations live most naturally in cusped manifolds.

Veering triangulations were introduced by Agol in \cite{Ago10}, where he canonically associated a veering triangulation to a pseudo-Anosov surface homeomorphism.
Let $g\colon  Y\to Y$ be a pseudo-Anosov map on a compact surface $Y$ with associated stable and unstable measured foliations $\F^s$ and $\F^u$. 
Let $Y'=\intr Y\setminus\{\text{singularities of $\F^s,\F^u$}\}$ and $g'= g|_{Y'}$, and let $M_g$ and $M_{ g'}$ be the respective mapping tori of $g$ and $ g'$. Agol constructs an ideal veering triangulation of $M_{g'}$ from a sequence of Whitehead moves between ideal triangulations of $Y'$ which are dual to a periodic splitting sequence of measured train tracks carrying the stable lamination of $g'$. Each Whitehead move corresponds to gluing a taut tetrahedron to $Y'$, and the resulting taut ideal triangulation of $Y'\times[0,1]$ glues up to give a taut ideal triangulation of $M_{g'}$. We call a taut ideal triangulation of this type \textbf{layered on $Y'$}. Agol shows that up to combinatorial equivalence there is only one veering triangulation of $M_{g'}$ which is layered on $Y$, and we call this the \textbf{veering triangulation of $g$}.

In \cite{Gue15} Gu\'eritaud provided an alternative construction of the veering triangulation of $g$ which we summarize now; a nice account is also given in \cite{MinTay17}.

Let $\wt{Y'}$ be the universal cover of $Y'$, and $\wh{Y'}$ be the space obtained by attaching a point to each lift of an end of $Y'$ to $\widetilde{Y'}$. The measured foliation $\F^s|_{Y'}$ lifts to $\wt{Y'}$ and gives rise to a measured foliation on $\wh{Y'}$ with singularities at points of $\wh{Y'}\setminus\wt{Y'}$. We call this the \textbf{vertical foliation}, and the analogous measured foliation coming from $\F^u|_{Y'}$ is called the \textbf{horizontal foliation}. In pictures, we will arrange the vertical and horizontal foliations so they are actually vertical and horizontal in the page.

The transverse measures on the vertical and horizontal foliations give $\wh{Y'}$ a singular flat structure.  A \textbf{singularity-free rectangle} is a subset of $\wh{Y'}$ which can be identified with $[0,1]\times[0,1]$ such that for all $t\in [0,1]$, $\{t\}\times[0,1]$ (resp. $[0,1]\times \{t\}$) is a leaf of the vertical (resp. horizontal) foliation.

We consider the family of singularity-free rectangles which are maximal with respect to inclusion. Any such maximal rectangle has one singularity in the interior of each edge. 

Each maximal rectangle $R$ defines a map $f_R\colon t_R\to R$ of a taut  tetrahedron into $R$ which ``flattens" $t_R$ and has the following properties: the pullback of the orientation on $R$ induces the correct coorientation on each triangle in $t_R$, the top (resp. bottom) $\pi$-edge of $t_R$ is mapped to a segment connecting the singularities on the horizontal (resp. vertical) edges of $R$, and $f(e)$  is a geodesic in the singular flat structure of $\wh{Y'}$. See Figure \ref{maximalrectangle}.

\begin{figure}[h]
\centering
\includegraphics[height=1.5in]{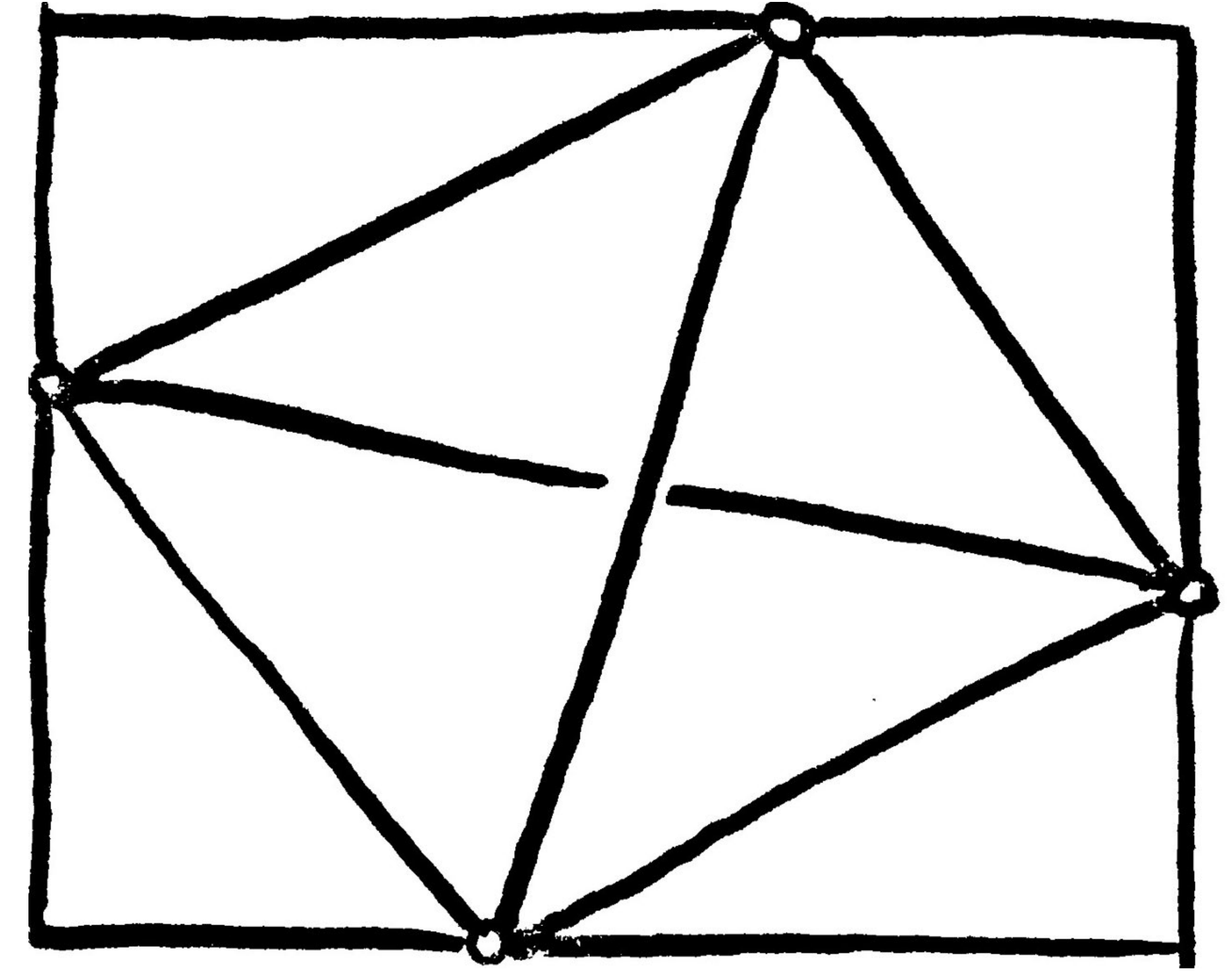}
\caption{A maximal rectangle and the image of the associated associated taut tetrahedron under flattening.}
\label{maximalrectangle}
\end{figure}

We can build a complex from $\bigcup t_{R_i}$, where the union is taken over all maximal rectangles, by making all the identifications of the following type: let $R_1$ and $R_2$ be maximal rectangles, and suppose $\Delta_1$ and $\Delta_2$ are faces of $t_{R_1}$ and $t_{R_2}$ such that $f_{R_1}(\Delta_1)=f_{R_2}(\Delta_2)$. Then we identify $\Delta_1$ and $\Delta_2$. The resulting complex is a taut ideal triangulation of $\wt{Y'}\times\R$, and one checks that it is veering. Gu\'eritaud showed that it descends to a layered veering triangulation of $M_{g'}$, which must be the one constructed by Agol.

While the veering triangulation is canonically associated to $g\colon Y\to Y$, it is in fact ``even more canonical" than that. To elaborate, we need the following result.

\begin{theorem}[Fried]\label{fried's theorem}
Let $\phi$ be a circular pseudo-Anosov flow on a compact 3-manifold $M$. Let $S$ be a cross section of $\phi$ and let $\sigma$ be the fibered face of $B_x(M)$ such that $[S]\in \cone(\sigma)$. An integral class $\alpha\in H_2(M,\partial M)$ lies in $\intr(\cone(\sigma))$ if and only if $\alpha$ is represented by a cross section to $\phi$. Up to reparameterization and conjugation by homeomorphisms of $M$ isotopic to the identity, $\phi$ is the only such circular pseudo-Anosov flow. 
\end{theorem}

 Thus we we will speak of \emph{the} \textbf{suspension flow of a fibered face}. Theorem \ref{fried's theorem} says that all the monodromies of fibrations coming from $\sigma$ are realized as first return maps of the suspension flow. Theorem \ref{fried's theorem} is proven by Fried in \cite{Fri79} in the case where $M$ is closed. The proof of this more general result, when the fiber possibly has boundary, does not seem to exist in the literature. For the reader's convenience we provide a proof in Appendix \ref{Fried appendix} (Theorem \ref{generalized Fried}).

Set $M=M_{g}$. Let $\sigma$ be the face of $B_x(M_{g})$ such that $[ Y]\in \cone (\sigma)$, let $\phi$ be the associated suspension flow, and let $\phi'=\phi|_{M_{g'}}$.

The lift $\wt{\phi'}$ of $\phi'$ to the universal cover $\wt{M_{g'}}$ of $M_{g'}$ is \textbf{product covered}, i.e. conjugate to the unit-speed flow in the $z$ direction on $\R^3$. Consequently the quotient of $\wt{M_{g'}}$ by the flowing action of $\R$, which we call the \textbf{flowspace of $\phi'$} and denote by $\flow(\phi')$, is homeomorphic to $\R^2$. In addition, $\flow(\phi')$ has 2 transverse (unmeasured) foliations which are the quotients of the lifts of the (2-dimensional) stable and unstable foliations of $\phi'$ to $\wt{M_{g'}}$.

Let $Z$ be another fiber of $M$ over $S^1$ such that $[Z]\in \cone(\sigma)$, with monodromy $h\colon Z\to Z$. Let $Z'$, $\wt{Z'}$, and $\wh{Z'}$ be obtained from $Z$ and $h$ in the same way that $Y'$, $\wt{Y'}$, and $\wh{Y'}$ were obtained from $Y$ and $g$.

By Theorem \ref{fried's theorem}, both $\wt{Y'}$ and $\wt{Z'}$ and their vertical and horizontal foliations can be identified with $\flow(\phi')$ and its foliations by forgetting measures. Hence we can identify $\wh{Y'}$ and $\wh{Z'}$ together with their vertical and horizontal foliations. The maximal rectangles from Gu\'eritaud's construction depend only on the vertical and horizontal foliations, and not on their transverse measures. The geodesics defining the edges of a tetrahedron do depend on the measures (and hence on $g$ and $h$), but for either pair of measures, the geodesics will be transverse to both the vertical and horizontal foliations. We see then that the triangles of the veering triangulation of $\wt{Y'}\times \R\cong \wt{Z'}\times \R\cong\wt{ M_{g'}}$ are well-defined up to isotopy. It follows that the veering triangulations of $g$ and $h$ are the same up to isotopy in $M_{g'}=M_{h'}=\intr(M)\setminus\{\text{singular orbits of $\phi$}\}$.
 
Synthesizing the above discussion, we have shown the following.

\begin{theorem}[Agol]\label{invarianceoftau}
Let $M$ be a compact hyperbolic 3-manifold, and suppose $Y$ and $Z$ are fibers of fibrations $M\to S^1$ with monodromies $g$ and $h$ such that $[Y],[Z]\in\cone(\sigma)$ for some fibered face of $B_x(M)$. Then the veering triangulations of $g$ and $h$ are combinatorially equivalent.

\end{theorem}

It therefore makes sense to speak of \emph{the} \textbf{veering triangulation of a fibered face} of $B_x(M)$.

\subsection{Some notation}\label{sec:notation}

We now fix some notation which will hold for the remainder of the paper.

Let $M$ be a compact hyperbolic 3-manifold, and let $\phi$ be a circular pseudo-Anosov flow on $M$.
 Let $c$ be the union of the singular orbits $c_1,\dots, c_n$ of $\phi$ and let $U_i$ be a small regular neighborhood of $c_i$. Let $V$ be a small regular neighborhood of $\partial M$ and put $U=V\cup \left(\bigcup_i U_i\right)$ and $\mr M= M\setminus U$. Let $\sigma$ be the (closed) fibered face of $B_x(M)$ determined by $\phi$, with associated veering triangulation $\tau$.

The homology long exact sequence associated to the triple $(M, U, \partial M)$ contains the sequence $H_2(U,\partial M)\xrightarrow{0} H_2(M,\partial M)\to H_2(M, U)$. By excision, $H_2(M, U)\cong H_2(\mr M, \partial\mr M)$. Hence there is an injective map 
\[
P\colon H_2( M, \partial M)\to H_2(\mr M, \partial\mr M). 
\]
At the level of chains, the map corresponds to sending a relative 2-chain $S$ to $S\setminus U$. We call $P$ the \textbf{puncturing map}, and if $\alpha\in H_2(M,\partial M)$ we often write $\mr\alpha$ to mean $P(\alpha)$.

\subsection{Veering triangulations, branched surfaces, and boundaries of fibered faces}\label{veering branched surface section}

Recall that a \textbf{branched surface} in $M$ is an embedded 2-complex transverse to $\partial M$ which is locally modeled on the quotient of a stack of disks $D_1,\dots, D_n$ such that for $i=1,\dots,n-1$, $D_i$ is glued to $D_{i+1}$ along the closure of a component of the complement of a smooth arc through $D_i$. The quotient is given a smooth structure such that the inclusion of each $D_i$ is smooth (see Figure \ref{branchedsurfacedef}). Branched surfaces were introduced in \cite{FloOer84} and experts will note that the above definition is different than the original; ours is the slightly more general one found in \cite{Oer86}.

\begin{figure}[h]
\centering
\includegraphics[height=1.5in]{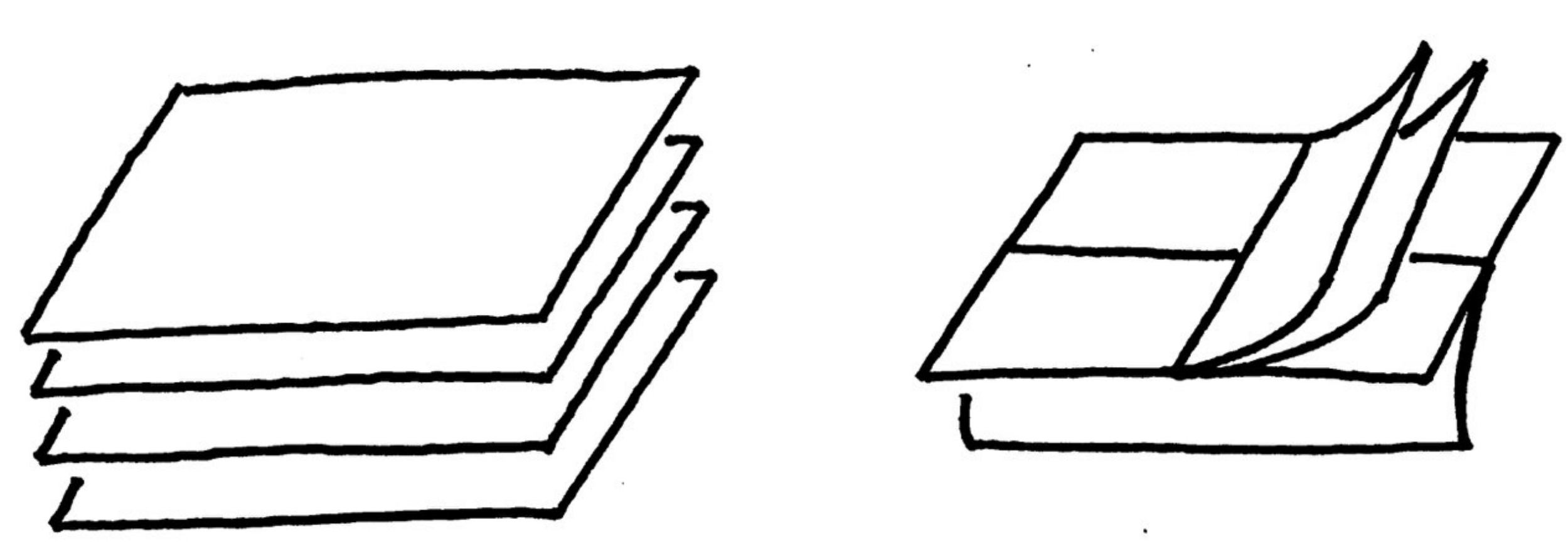}
\caption{A portion of a branched surface, obtained as the quotient of a stack of disks.}
\label{branchedsurfacedef}
\end{figure}

Let $B\subset M$ be a branched surface. The union of points in $B$ with no neighborhood homeomorphic to a disk is called the \textbf{branching locus} of $B$, and the components of the complement of the branching locus are called \textbf{sectors}. Let $N(B)$ be the closure of a regular neighborhood of $B$. Then $N(B)$ can be foliated by closed intervals transverse to $B$, and we call this the \textbf{normal foliation} of $N(B)$. If the normal foliation is oriented we say $B$ is an \textbf{oriented} branched surface. In this paper all  branched surfaces will be oriented. If $S\subset N(B)$ is a cooriented surface properly embedded in $M$ which is positively transverse to the normal foliation in the sense that its coorientation is compatible with orientation of the normal foliation, we say $S$ is \textbf{carried by} $B$. A surface carried by $B$ gives a system of nonnegative integer weights on the sectors of $B$ which is compatible with natural linear equations along the branching locus of $B$. Any system of nonnegative real weights $w$ satisfying these linear equations gives a 2-cycle and thus determines a homology class  $[w]\in H_2(M,\partial M)$. We say $[w]$ is \textbf{carried by} $B$. The collection of homology classes carried by $B$ is clearly a convex cone.

Put $\mr\tau=\tau\cap \mr M$. Then the 2-skeleton of $\mr\tau$ has the structure of a cooriented branched surface. We call this branched surface $B_{\mr\tau}$.

\begin{corollary}[Agol]\label{VTcarriesface}
Let $\sigma$ be the closed face of $B_x(M)$ determined by $\phi$, and let $\mr\sigma$ be the face of $B_x(\mr M)$ such that $P(\sigma)\subset\cone(\mr\sigma)$. The cone of classes in $H_2(\mr M, \partial\mr M)$ carried by $B_{\mr\tau}$ is equal to $\cone(\mr\sigma)$.
\end{corollary}

\begin{proof}
Let $\Sigma$ be a fiber of $M$ such that $[\Sigma]\in\cone(\mr\sigma)$. By Theorem \ref{invarianceoftau}, $\tau$ can be built as a layered triangulation on $\intr(\Sigma)$. It follows that $[\Sigma]$ is carried by $B_{\mr\tau}$. Therefore any integral class in $\intr(\cone(\mr\sigma))$ is carried, so every rational class in $\intr(\cone(\mr\sigma))$ is carried. 

We can find a closed oriented transversal through each point of $B_{\mr\tau}$, so $B_{\mr\tau}$ is a \textbf{homology branched surface} in the sense of \cite{Oer86}; it follows that the cone of classes it carries is closed (see \cite{Lan18} for details). Since each class in $\cone(\mr\sigma)$ is approximable by a sequence of rational classes in the interior of the cone, the proof is finished.
\end{proof}

\subsection{Relating $\partial\mr\tau$ and $\phi$}

\begin{figure}[h]
\centering
\includegraphics[height=1in]{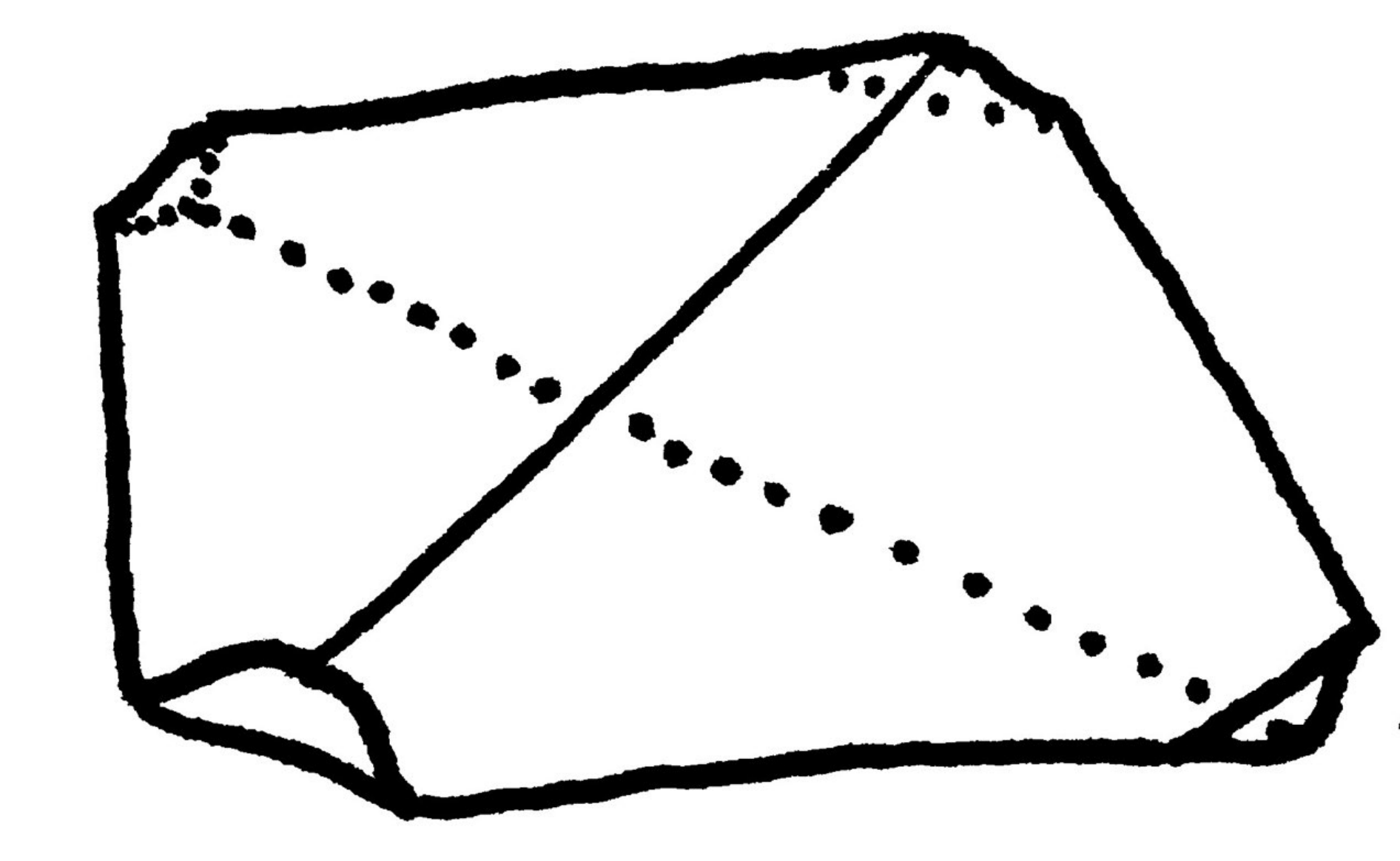}
\caption{A truncated taut tetrahedron. The coorientation is such that the solid edge with angle $\pi$ is on top.}
\label{truncatedtauttetrahedron}
\end{figure}

The complex $\mr\tau$ is a decomposition of $\mr M$ into truncated taut tetrahedra. Let $t$ be a taut tetrahedron of $\tau$, and let $t'$ be the truncation of $t$ obtained by deleting $t\cap U$. The \textbf{top}, \textbf{bottom}, and \textbf{top} and \textbf{bottom $\pi$-edges of $t'$} are defined to be the restrictions to $t'$ of the corresponding parts of $t$.
Each of the four faces of $t'$ corresponding to ideal vertices of $t$ is a triangle with two interior angles of 0 and one of $\pi$, which we call a \textbf{flat triangle}.
Because the faces of $t$ are cooriented, a flat triangle inherits a coorientation on its edges. We say a flat triangle which is cooriented outwards (respectively inwards) at its $\pi$ vertex is an \textbf{upward} (resp. \textbf{downward}) flat triangle, as in Figure \ref{upwardanddownwardtriangles}. The upwards flat triangles of $t'$ correspond to the ideal vertices connected by the top $\pi$-edge of $t'$.

\begin{figure}[h]
\centering
\includegraphics[height=.4in]{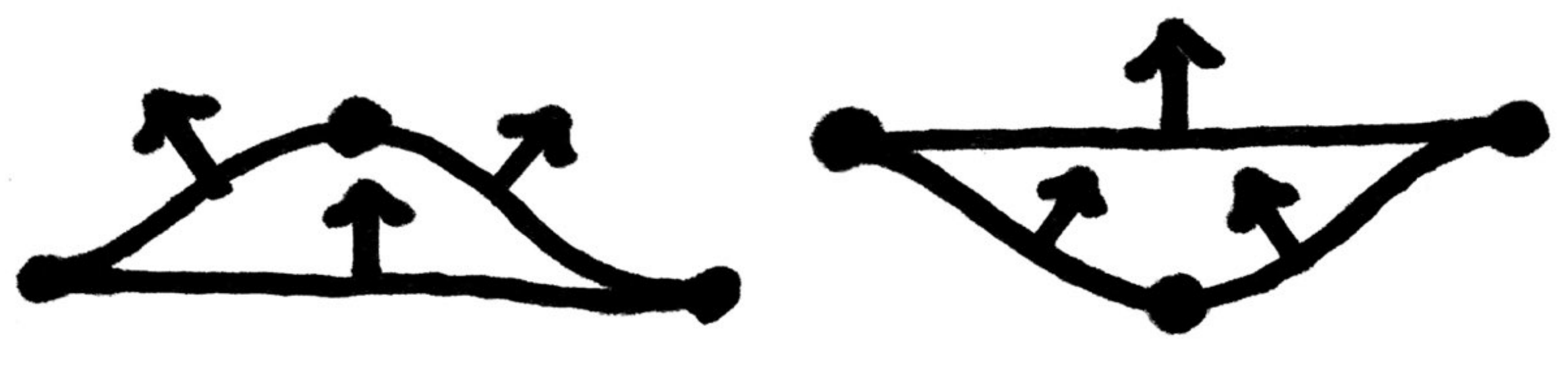}
\caption{Upward (left) and downward (right) flat triangles.}
\label{upwardanddownwardtriangles}
\end{figure}

The flat triangles of $\mr \tau$ give a triangulation of $\partial M$. Some of the combinatorics of this triangulation are described in \cite{Gue15} and \cite{Lan18}.

The union $u$ of all upward flat triangles is a collection of annuli. The triangulation restricted to each annulus component of $u$ has the property that each edge of a flat triangle either traverses the annulus or lies on the boundary of the annulus. The same holds for the union $d$ of all downward flat triangles. A component of $u$ is called an \textbf{upward ladder} and a component of $d$ is called a \textbf{downward ladder}. We call the boundary components of ladders \textbf{ladderpoles}. 
\begin{figure}[h]
\centering
\includegraphics[height=4.4in]{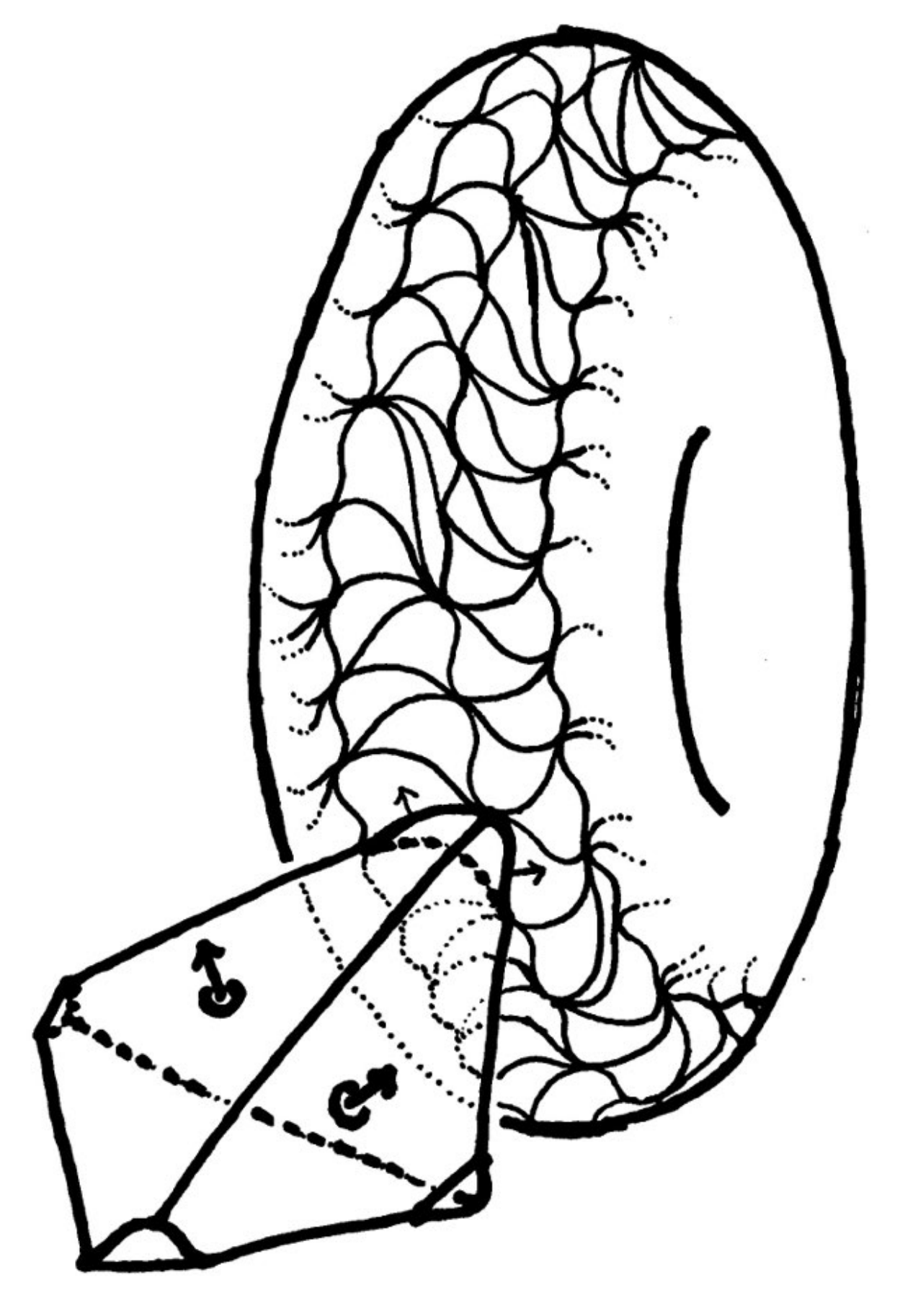}
\caption{A truncated taut tetrahedron meeting a component of $\partial\mr M$ in an upward flat triangle. An upward and downward ladder are shown. The top $\pi$-edge of the truncated taut tetrahedron must be left veering since it corresponds to a vertex in the right ladderpole of an upward ladder.}
\label{tetrahedronandtraintrack}
\end{figure}

The 1-skeleton of this triangulation by flat triangles of $\partial \mr M$ is a cooriented train track, which we call $\partial \mr \tau$.
We define a notion of left and right on each branch $e$ of $\partial\mr\tau$: orient $\partial M$ inwards, and map some neighborhood of $e$ homeomorphically to $\R^2$ so that $e$ is identified with $[-1,1]\times \{0\}$ and the pushforward of $e$'s coorientation points in the positive $y$ direction. Define the \textbf{left} (resp. \textbf{right}) \textbf{switch of} $e$ to be the preimage of $-1$ (resp. 1). We orient each branch of $\partial\mr\tau$ from right to left, and this consistently determines an orientaion on $\partial\mr\tau$. If an oriented curve is carried by $\partial\mr\tau$ such that orientations agree, we say $\gamma$ is \textbf{positively carried} by $\partial\mr\tau$. By our choice of orientation if $S$ is a surface carried by $B_{\mr\tau}$ then its boundary, given the orientation induced by an outward-pointing vector field on $ \partial\mr M$, is positively carried by $\partial\mr\tau$.

We call branches of $\partial\mr\tau$ contained in ladderpoles \textbf{ladderpole branches}, and branches that traverse ladders \textbf{rungs}.
We define the \textbf{left} (resp. \textbf{right}) \textbf{ladderpole} of a ladder to be the ladderpole containing the left (resp. right) boundary switches of its rungs. A closed oriented curve positively carried by $\partial \mr\tau$ and traversing only ladderpole branches is called a \textbf{ladderpole curve}.

Note that a switch of $\partial\mr\tau$ corresponds to an edge of $\tau$. The combinatorics of the flat and veering triangulations are related by the following Lemma, the proof of which is elementary.

\begin{lemma}\label{lem:veeringrule}
Let $v$ be a switch of $\partial\mr\tau$ corresponding to an edge $e$ of $\tau$. If $v$ lies in the left ladderpole of an upward ladder then $e$ is right veering. If $v$ lies in the right ladderpole of an upward ladder then $e$ is left veering.
\end{lemma}

Consider a singular leaf $L$ of the stable foliation of $\phi$ meeting a singular orbit $c_i$ of $\phi$. Topologically $L$ is the quotient of a half-closed annulus $A$ by a map which wraps $\partial A$ around $c_i$ some finite number of times.  The flow lines of $\phi|_L$ converge in forward time to $c_i$. Since $L$ is dense in $M$, the intersection $L\cap \closure(U_i)$ has many components; we will call the component containing $c_i$ a \textbf{stable flow prong} of $c_i$. An \textbf{unstable flow prong} of $c_i$ is defined symmetrically with a singular leaf of the unstable foliation of $\phi$.

\begin{lemma}\label{ladderseparatrixstructure}
A stable flow prong of $c_i$ intersects $\partial U_i$ in the interior of an upward ladder. Symmetrically, an unstable 2-prong of $c_i$ intersects $\partial U_i$ in the interior of a downward ladder.
\end{lemma}

\begin{figure}[h]
\centering
\includegraphics[height=2.5in]{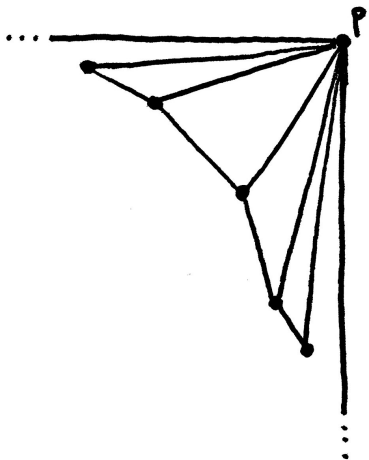}
\caption{Four triangles of $\wh T$, lying in a single quadrant bounded by vertical and horizontal leaves which meet at the singular point $p$ of $\wh {Y'}$.}
\label{fig_ladderseparatrixstructure}
\end{figure}

\begin{proof}
Fix a fiber $Y$ of $M$ with $[Y]\in\cone(\sigma)$, and let $\wt{Y'}$ and $\wh{Y'}$ be as in Section \ref{canonical veering section}. Let $M'=\intr(M)\setminus \{\text{singular orbits of $\phi$}\}$.

Pick a singular orbit $c_i$ of $\phi$ and consider the left ladderpole $\ell$ of an upward ladder of $U_i$. The vertices of $\ell$ define a family $T$ of triangles of $\tau$. Let $T'=T\cap U_i$, let $\wt {T'}$ be a component of the lift of $T'$ to $\wt{M'}$, the universal cover of $M'$, let $\wt T$ be the union of triangles in the veering triangulation of $\wt{M'}$  intersecting $\wt{T'}$, and let $\wh T$ be the closure in $\wh{Y'}$ of the projection of $\wt T$. This is a bi-infinite sequence of triangles, each sharing an edge with the next, and all sharing a single vertex at a singular point $p$ of $\wh{Y'}$.

The vertical and horizontal foliations define infinitely many quadrants each meeting $p$ in a corner of angle $\frac{\pi}{2}$, which are each bounded by one vertical and one horizontal leaf. We claim $\wh T$ lives in only one of these quadrants, a situation depicted in Figure \ref{fig_ladderseparatrixstructure}.

In $\wh{Y'}$, the edges of positive slope are right veering while the edges of negative slope are left veering. Hence every edge in $\wh T$ meeting $p$ lies in a quadrant with horizontal left boundary leaf and vertical right boundary leaf (where our notion of left and right is determined looking at $p$ from inside $\wt{Y'}$).

Fix one of the triangles of $\wh T$, defined by edges $e_1$ and $e_2$ meeting $p$. The two edges lie in the interior of a singularity-free rectangle with vertical and horizontal sides containing $p$ in its boundary. Such a rectangle lies in the union of two adjacent quadrants of $p$, only one of which can have horizontal left boundary and vertical right boundary. Therefore $e_1$ and $e_2$ lie in the same quadrant, and all the edges of $\wh T$ lie in the same quadrant by induction.

Applying this analysis to each ladderpole yields the result.
\end{proof}

\begin{figure}[h]
\centering
\includegraphics[height=2.5in]{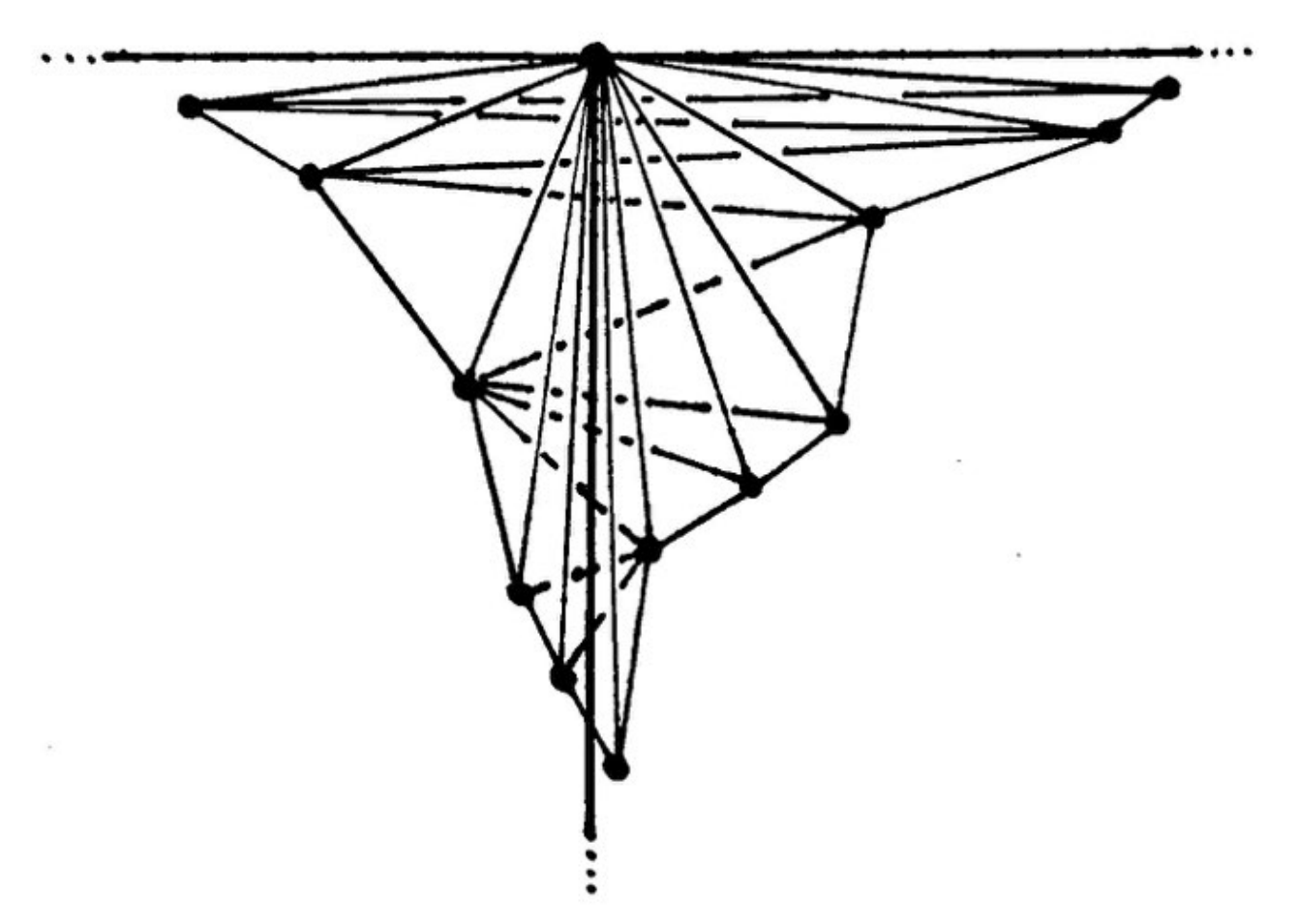}
\quad
\includegraphics[height=2in]{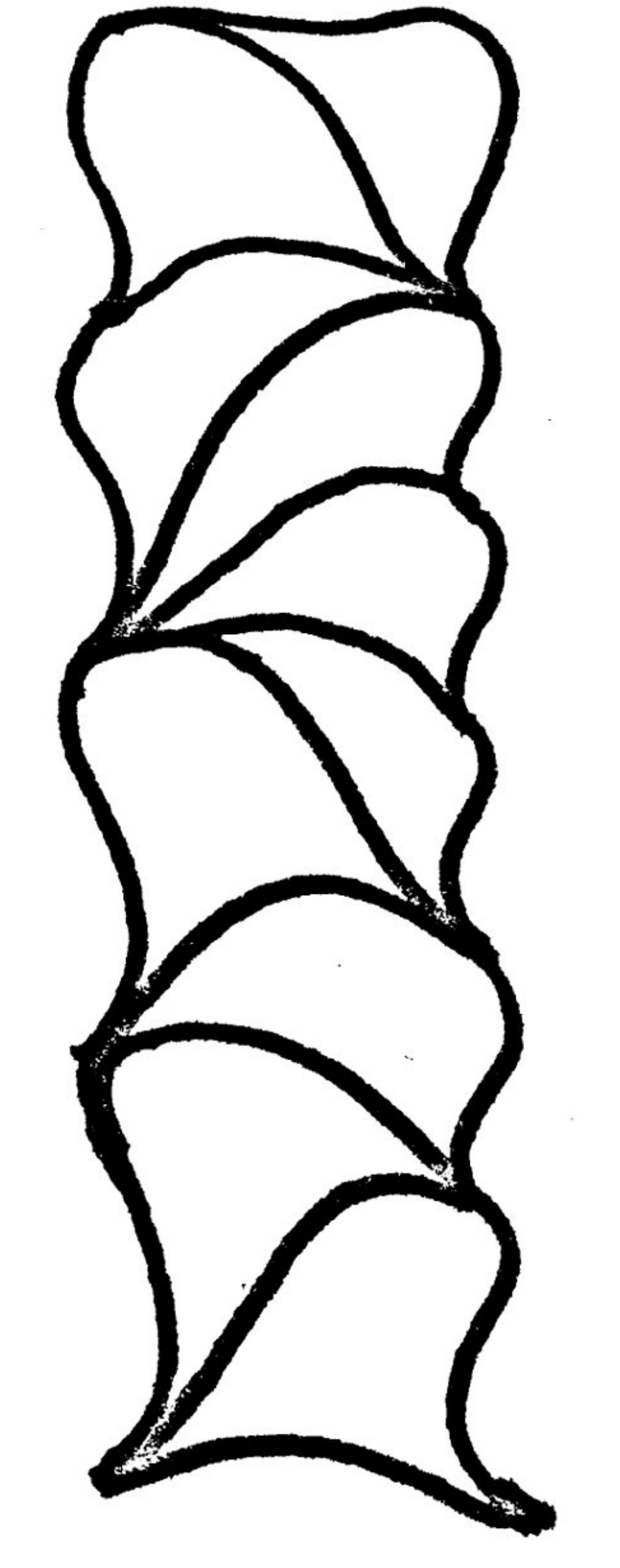}
\caption{These 10 taut tetrahedra (left) define this part of an upward ladder (right) on the torus coming from the top, central singularity on the left.
}
\label{upwardladder}
\end{figure}

\section{Almost transverse surfaces}

\subsection{Dynamic blowups}

We now describe the process of dynamically blowing up a singular periodic orbit $\gamma$ of a pseudo-Anosov flow $\phi$, which can be thought of as replacing a singular orbit by the suspension of a homeomorphism of a tree. For more details, the reader can consult \cite{Mos92b, Mos91, Mos90}.

Let $q\in \mathbb N$, $q\ge 3$. Define a \textbf{pseudo-Anosov star} to be a directed tree $T$ embedded in the plane with $2q$ edges meeting at a central vertex $v$, such that the orientations of edges around $v$ alternate between inward and outward with respect to $v$. We say a directed tree $T^\sharp$ is a \textbf{dynamic blowup} of $T$ if the closed neighborhood of each vertex of $T^\sharp$ is a pseudo-Anosov star, and there exists a cellular map $\pi\colon T^\sharp\to T$ preserving edge orientations such that $\pi$ is injective on the complement of $\pi^{-1}(v)$. See Figure \ref{exampleblowups} for two examples.

\begin{figure}[h]
\centering
\includegraphics[height=1.5in]{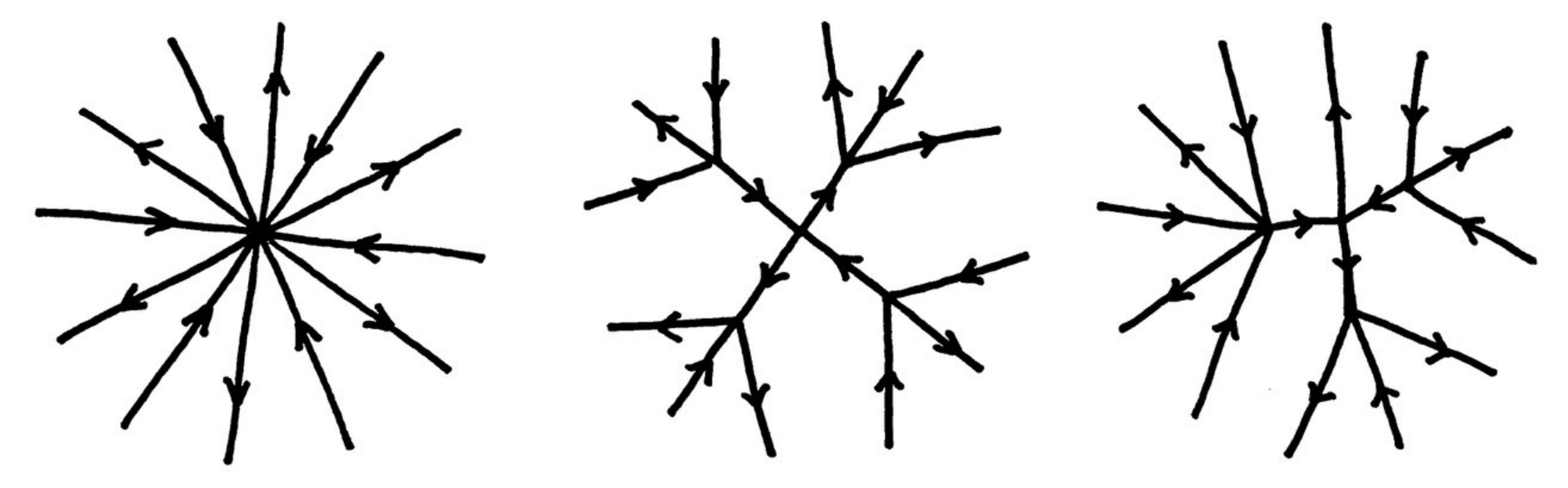}
\caption{A pseudo-Anosov star (left) and 2 possible dynamic blowups}
\label{exampleblowups}
\end{figure}

Let $\gamma$ be a singular periodic orbit of $\phi$ meeting $q$ stable and $q$ unstable flow prongs, and suppose $\phi$ rotates the flow prongs by $2\pi \cdot \frac{p}{q}$ traveling once around $\gamma$, where $\gcd(p,q)=1$. 

The intersection of the flow prongs of $\gamma$ with a local cross section of $\phi$ gives a pseudo-Anosov star $T$ with $2q$ edges, with each edge oriented according to whether points in that flow prong spiral towards or away from $\gamma$ in forward time. Let $T^\sharp$ be a dynamic blowup of $T$ that is invariant under rotation by $2\pi \cdot \frac{p}{q}$, and let $G$ be the preimage of the central vertex of $T$ under the collapsing map $\pi\colon T^\sharp\to T$.

There is a flow $\phi^\sharp$ on $M$ which replaces $\gamma$ by the suspension of a homeomorphism $h\colon G\to G$ with the following properties. Each edge $E$ of $G$ is mapped by $h$ to its image under rotation by $2\pi \cdot \frac{p}{q}$, and $h^q(E)$ fixes vertices and moves interior points in the direction $E$ inherits from $T^\sharp$.

The orbit of $G$ under $\phi^\sharp$ is a complex of annuli $A$ invariant under $\phi^\sharp$. A point interior to an annulus of $A$ spirals away from one boundary circle and towards the other in forward time, and the boundaries of annuli are closed orbits of $\phi^\sharp$. The flows $\phi^\sharp$ and $\phi$ are semiconjugate via a map collapsing $A$ to $\gamma$ which is injective on the complement of $A$.

We say $\phi^\sharp$ is obtained from $\phi$ by \textbf{dynamically blowing up} $\gamma$. A flow obtained from $\phi$ by dynamically blowing up some collection of singular orbits is a \textbf{dynamic blowup} of $\phi$. 

The effect of dynamically blowing up a singular orbit is to pull apart its flow prongs, otherwise leaving the dynamics of the flow unchanged. See Figure \ref{3dblowup}.

\begin{figure}[h]
\centering
\includegraphics[height=2in]{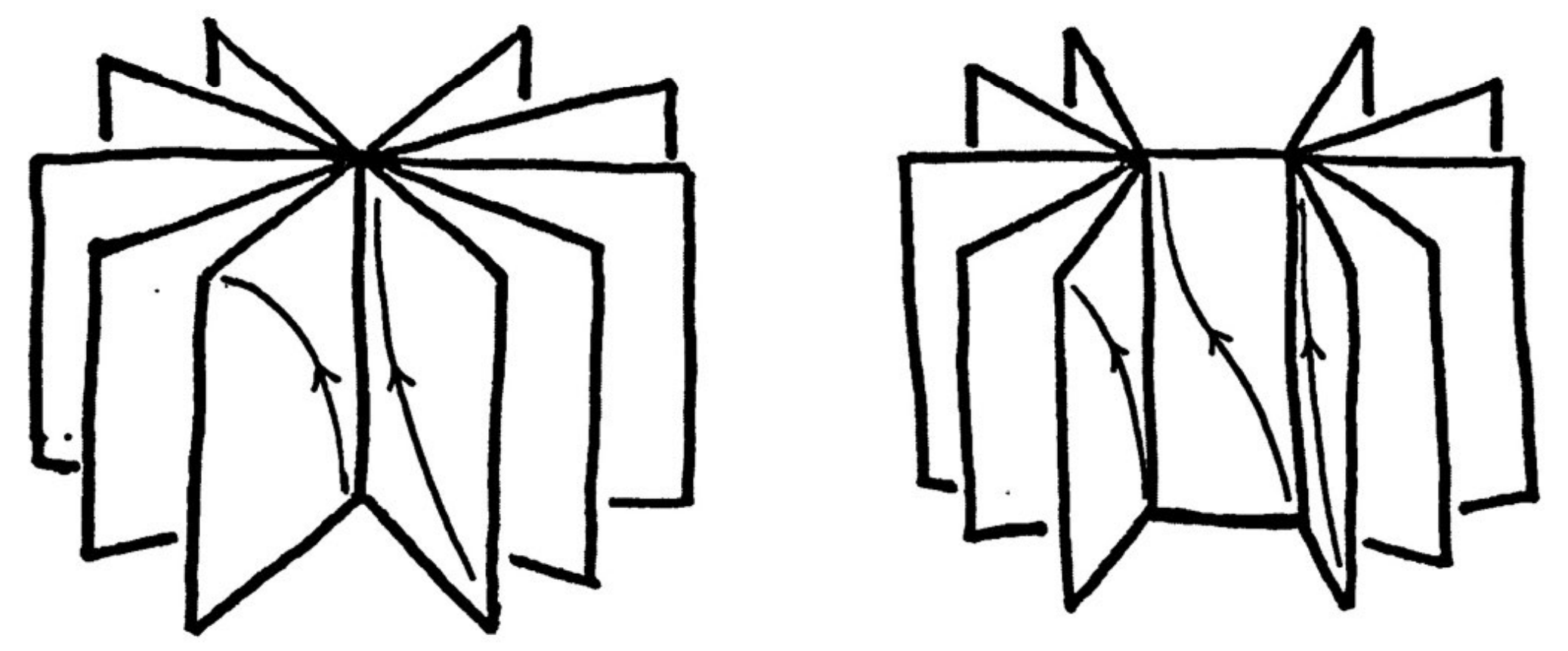}
\caption{A dynamic blowup pulls apart the flow prongs of a singular orbit .}
\label{3dblowup}
\end{figure}

\subsection{Statement of Transverse Surface Theorem, and Mosher's approach}\label{TSTstatement}

\begin{definition}
Let $\psi$ be a circular pseudo-Anosov flow on a 3-manifold $N$. We say an oriented surface $S$ embedded in $N$ is \textbf{almost transverse to $\psi$} if there is a dynamic blowup $\psi^\sharp$ of $\psi$ such that $S$ is transverse to $\psi^\sharp$ and the orientation of $TS\oplus T\psi^\sharp$  agrees with that of the tangent bundle of $N$ at every point in $N$.
\end{definition}

\begin{TST}[Mosher]
Let $N$ be a closed hyperbolic 3-manifold with fibered face $F$ and associated suspension flow $\psi$. An integral class $\alpha\in H_2(N)$ lies in $\cone(F)$ if and only if it is represented by a surface which is almost transverse to $\psi$.
\end{TST}

Let $S$ be a surface which is almost transverse to $\psi$, and thus transverse to some dynamic blowup $\psi^\sharp$ of $\psi$. Then $S$ is taut, and $[S]\in\cone(F)$. This is true even when $N$ has boundary, as we will now show.

Note that since $\psi$ is a topologically transitive flow, meaning it has a dense orbit, $\psi^\sharp$ is as well. Consider a point $s\in S'$ for some component $S'$ of $S$. There is an open neighborhood $\epsilon(s)$ of $s$ which is homeomorphic to $(0,1)^3$ such that the restricted flow lines of $\psi^\sharp$ correspond to the vertical lines $\{a\}\times\{b\}\times(-1,1)$ for $-1<a,b<1$, and $S'\cap\epsilon(s)$ corresponds to $(-1,1)^2\times \{0\}$. Let $o$ be a dense orbit of $\psi^\sharp$. We can take a segment of $o$ with endpoints near each other in $\epsilon(s)$ and attach endpoints with a short path to obtain a closed curve $\gamma(o)$ positively intersecting $S'$, so $S'$ is homologically nontrivial. In particular it follows that $S$ has no sphere or disk components.

We record a lemma:

\begin{lemma}
Let $\phi^\sharp$ be a dynamic blowup of $\phi$. Then the tangent vector fields of $\phi$ and $\phi^\sharp$ are homotopic, and consequently $e_\phi=e_{\phi^\sharp}$.
\end{lemma}

Next, note that the restriction of $\xi_{\psi^\sharp}$ to $S$ is homotopic to $TS$, so $\langle e_{\psi^\sharp},[S]\rangle=\chi([S])$. Here $\langle \,\cdot\,,\,\cdot \,\rangle$ denotes the pairing of $H^2$ with $H_2$. Hence
\begin{align*}
x([S])&\le -\chi(S)\\
&=\langle-e_{\psi^\sharp}, [S]\rangle\\
&=\langle -e_{\psi},[S]\rangle \\
&\le x([S]).
\end{align*}

The first inequality holds because $S$ has no sphere or disk components. The final inequality follows from the fact that $x$ is a supremum of finitely many linear functionals on $H_2(N)$ which include $\langle -e_\phi,\, \cdot \,\rangle$.
The fact that $x([S])= -\chi(S)$ means that $S$ is taut, while $\langle  -e_\phi, [S]\rangle$ means that $x$ and $\langle -e_\phi,\, \cdot \,\rangle$ agree on $[S]$ so $[S]\in \cone(\sigma)$.

In light of the above discussion, to prove the Transverse Surface Theorem it suffices to produce an almost transverse representative of any integral class in $\cone(F)$, and since any such class in the interior of $\cone(F)$ is represented by a cross section, it suffices to produce an almost transverse representative for any integral class in $\partial\cone(F)$.

Mosher's proof of the Transverse Surface Theorem spans \cite{Mos89, Mos90, Mos91}; we give a brief summary here. 
Given an integral class $\alpha\in H_2(N)$ lying in $\partial\cone(\sigma)$, we consider its Poincar\'e dual $u\in H^1(N)$. Associated to $u$ is an infinite cyclic covering space $N_\Z$. Mosher shows that there is a way to dynamically blow up a collection of singular orbits of $\psi$ to get a dynamic blowup $\psi^\sharp$ that lifts to a flow $\tilde{\psi}^\sharp$ on $M_\Z$ with nice dynamics. More specifically, he defines a natural partial order $\le$ on the set of chain components of $\tilde{\psi}^\sharp$ and shows that $\tilde{\psi}^\sharp$ has finitely many chain components up to the deck action of $\Z$. He constructs a strongly connected directed graph $\Gamma$ with vertices the deck orbits of chain components of $\tilde{\psi}^\sharp$, and edges determined by $\le$. He shows that flow isotopy classes of surfaces transverse to $\psi^\sharp$ and compatible with $u$ are in bijection with positive cocycles on $\Gamma$ representing a cohomology class $v\in H^1(\Gamma)$ which is determined by $u$. Finally he proves the existence of such a cocycle.

\subsection{A veering proof of the Transverse Surface Theorem}

The proof of the Transverse Surface Theorem which we present in this section depends on the combinatorial Lemma \ref{blowupfill}, the statement of which requires some definitions.

A \textbf{pseudo-Anosov tree} $T$ is a directed tree such that the closed neighborhood of each non-leaf vertex is a pseudo-Anosov star. We can embed $T$ in a disk $\Delta$ such that its leaves lie in $\partial \Delta$. 

Let $d$ be the closure of a component of $\Delta\setminus T$. Then $d$ is homeomorphic to a closed disk and $\partial d$ is a union of  edges in $T$ and one closed interval $I$ in $\partial\Delta$. The boundary of $I$ is composed of two leaves of $T$. Each leaf can be assigned a $+$ or $-$ depending on whether the corresponding edge of $T$ points into or away from the leaf, respectively. We endow $I$ with the orientation pointing from $-$ to $+$.

\begin{lemma}\label{regionalblowup}
Let $A,B$ be complementary regions of $T$ which are incident along a single vertex of $T$, and the orientations on $\closure(A)\cap\partial\Delta$ and $\closure(B)\cap\partial\Delta$ are opposite. Then there is a unique dynamic blowup $T^\sharp$ of $T$ with one more edge than $T$ such that $A$ and $B$ are incident along an edge of $T^\sharp$.
\end{lemma}

A picture makes this obvious; in lieu of a proof, see Figure \ref{blowupbyregions} for a diagram of this dynamic blowup (for simplicity, we are abusing notation slightly by identifying $A$ and $B$ with the corresponding complementary regions of $T^\sharp$).

\begin{figure}[h]
\centering
\includegraphics[height=1.5in]{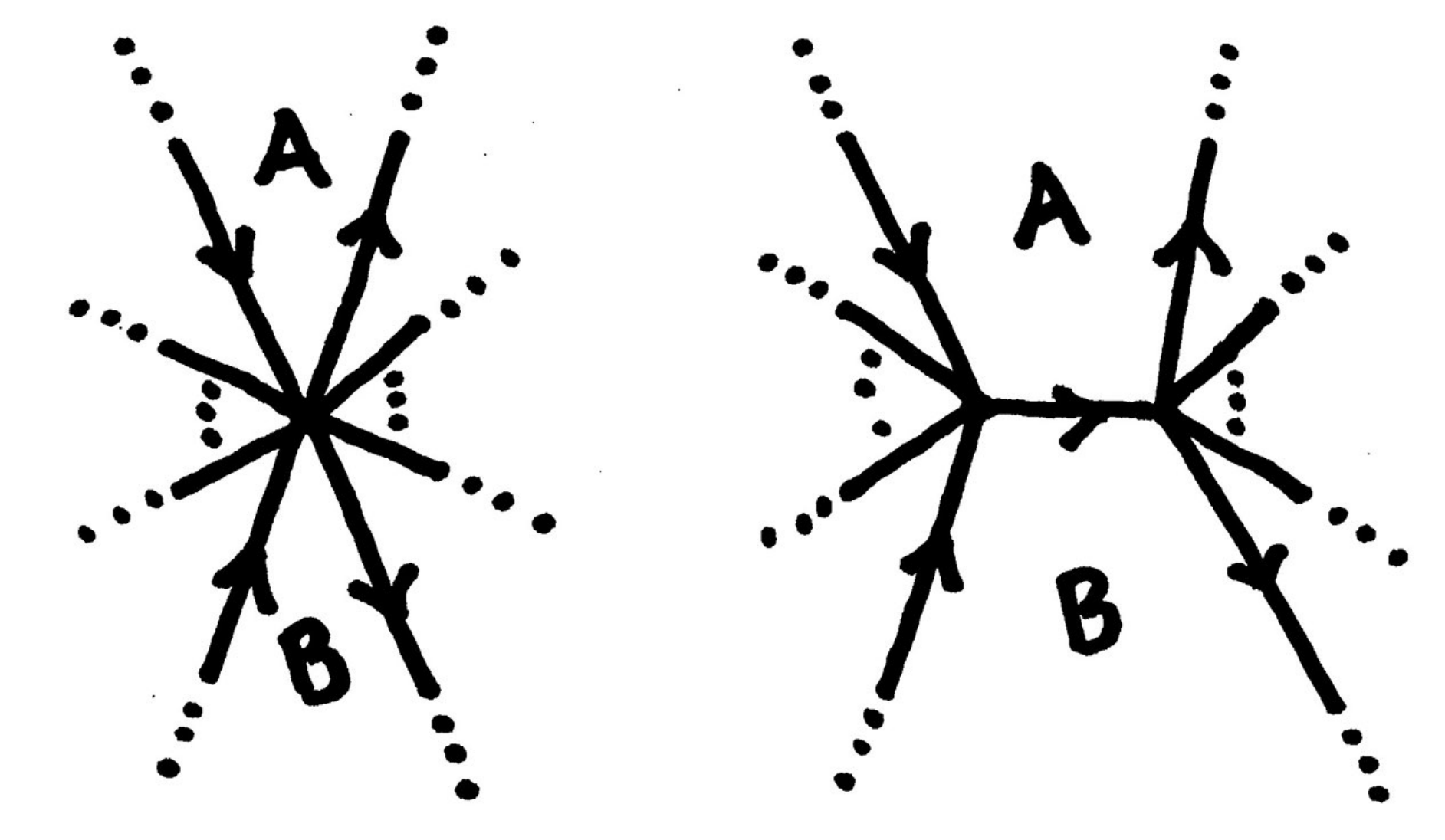}
\caption{}
\label{blowupbyregions}
\end{figure}

If we choose an orientation for $\partial\Delta$, then any point $p$ lying in a component $I$ of $\partial\Delta\setminus T$ can be given a sign according to whether the orientation of $\Delta$ agrees with the orientation of the component of $\partial\Delta\setminus T$ containing $p$.

An \textbf{even family} $E$ for a pseudo-Anosov tree $T$ is a finite subset of $\partial \Delta \setminus T$ which represents 0 in $H_0(\partial\Delta)$ when each $p\in E$ is given a sign as in the previous paragraph and $E$ is viewed as a 0-chain. We say an even family can be \textbf{filled in over $T$} if there exists a family $L$ of disjoint cooriented line segments with $L\cap \partial\Delta=\partial L$ such that:

\begin{itemize}
\item $\partial L=E$ in the cooriented sense, i.e. the coorientations at each point in $\partial L$ agree with the orientation of each segment of $\partial\Delta\setminus T$, and 

\item for each segment $\ell$ of $L$, $\ell$ intersects $T$ only transversely in the interior of edges such that the coorientation of $\ell$ agrees with the orientation of the intersected edge.
\end{itemize}

\begin{figure}[h]
\centering
\includegraphics[height=2in]{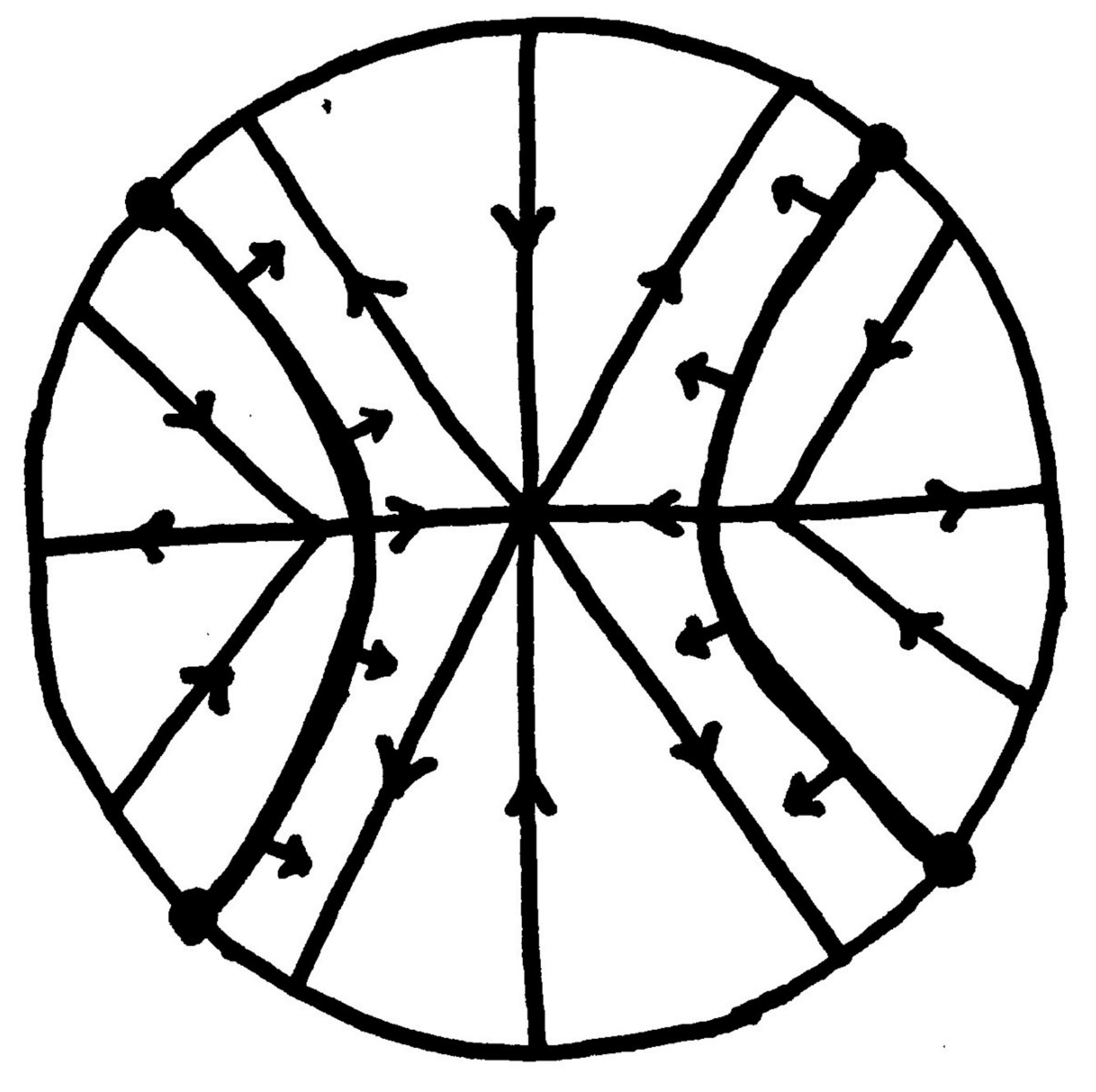}
\caption{In this picture we see a pseudo-Anosov tree $T$ and an even family of size 4 being $\pi$-symmetrically filled in over $T$ by two cooriented line segments.}
\label{fillingin}
\end{figure}

If $T$ and $E$ above are symmetric under rotation of $\Delta$ by an angle of $\theta$, and $L$ can be chosen to respect this symmetry, we say $E$ can be \textbf{$\theta$-symmetrically filled in over $T$}.  Figure \ref{fillingin} shows an example.

\begin{lemma}\label{blowupfill}
Let $S$ and $E$ be a pseudo-Anosov star and an even family for $S$ respectively that are symmetric under rotation by $\theta$. There exists a dynamic blowup $S^\sharp$ of $S$ such that $E$ can be $\theta$-symmetrically filled in over $S^\sharp$.
\end{lemma}

Note that the lemma statement includes the case $\theta=0$.

\begin{figure}[h]
\centering
\includegraphics[height=1.9in]{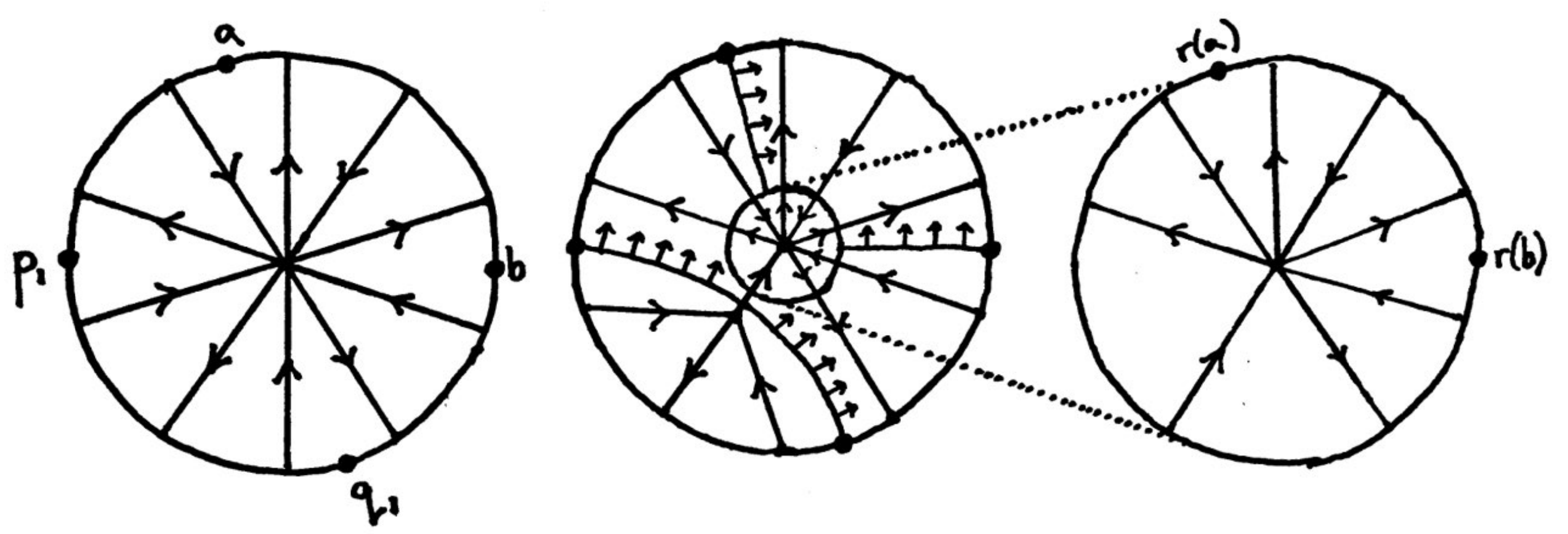}
\caption{A diagram of the proof of Lemma \ref{blowupfill} when $S$ is the pseudo-Anosov star on the left with an even family of size 4, and $\theta=0$. In the notation of the proof, $E=\{p_1,q_1,a,b\}$ and $E_1=\{r(a),r(b)\}$.}
\label{blowupfillfigure}
\end{figure}

\begin{proof}
Choose a pair of points $p_1,q_1$ in $E$ of opposite sign which are circularly adjacent and let $p_1,\dots,p_n$ and $q_1,\dots,q_n$ be all their images, without repeats, under rotation of $\Delta$ by $\theta$. Let $P_i$, $Q_i$ be the components of $\Delta\setminus S$ corresponding to $p_i, q_i$ respectively. If $P_i$ and $Q_i$ are incident along an edge of $S$ then there is a family of cooriented line segments filling in $\{p_i,q_i\}_{i=1}^n$ over $S$. Otherwise  $P_i$ and $Q_i$ are incident at the vertex of $S$ and determine a dynamic blowup of $S$ as in Lemma \ref{regionalblowup}. Since the pairs $(p_i, q_i)$ are unlinked in $\partial\Delta$, we may perform all $n$ of these dynamic blowups in concert to obtain a dynamic blowup $S^\sharp$ of $S$ such that $\{p_i,q_i\}_{i=1}^n$ can be $\theta$-symmetrically filled in over $S^\sharp$.

Let $L$ be the family of cooriented line segments filling in $\{p_i,q_i\}_{i=1}^n$. Let $E'=E\setminus \{p_i,q_i\}_{i=1}^n$. By construction, $E'$ is contained in a single component $\Delta'$ of $\Delta\setminus L$, and $S^\sharp\cap \Delta'$ has a single vertex $v$ which is preserved under rotation by $\theta$. There exists a closed disk $\Delta_1\subset \Delta'$ centered around $v$ which is also preserved under rotation by $\theta$. We can connect each point $e$ in $E'$ by a cooriented line segment to its image $r(e)$ under a retraction $r\colon \Delta\to\Delta_1$ such that the union of these segments is invariant under rotation by $\theta$. Let $E_1=\{d(e)\mid e\in E'\}$, and $S_1=S^\sharp\cap \Delta_1$.
A picture of this situation is shown in Figure \ref{blowupfillfigure}.

We see that $E_1$ is an even family for $S_1$ which is smaller than $E$. Iterating this procedure, we eventually $\theta$-symmetrically fill in $E$ over a dynamic blowup of $S$.
\end{proof}

Now we are equipped to prove the Transverse Surface Theorem. We actually will prove a generalization to compact manifolds which might have boundary.

\begin{theorem}[Almost transverse surfaces]\label{myTST}
Let $M$ be a compact hyperbolic 3-manifold, with a fibered face $\sigma$ of $B_x(M)$ and associated circular pseudo-Anosov suspension flow $\phi$. Let $\alpha\in H_2(M,\partial M)$ be an integral homology class. Then $\alpha\in \cone(\sigma)$ if and only if $\alpha$ is represented by a surface almost transverse to $\phi$.
\end{theorem}

\begin{figure}[h]
\centering
\includegraphics[height=2.5in]{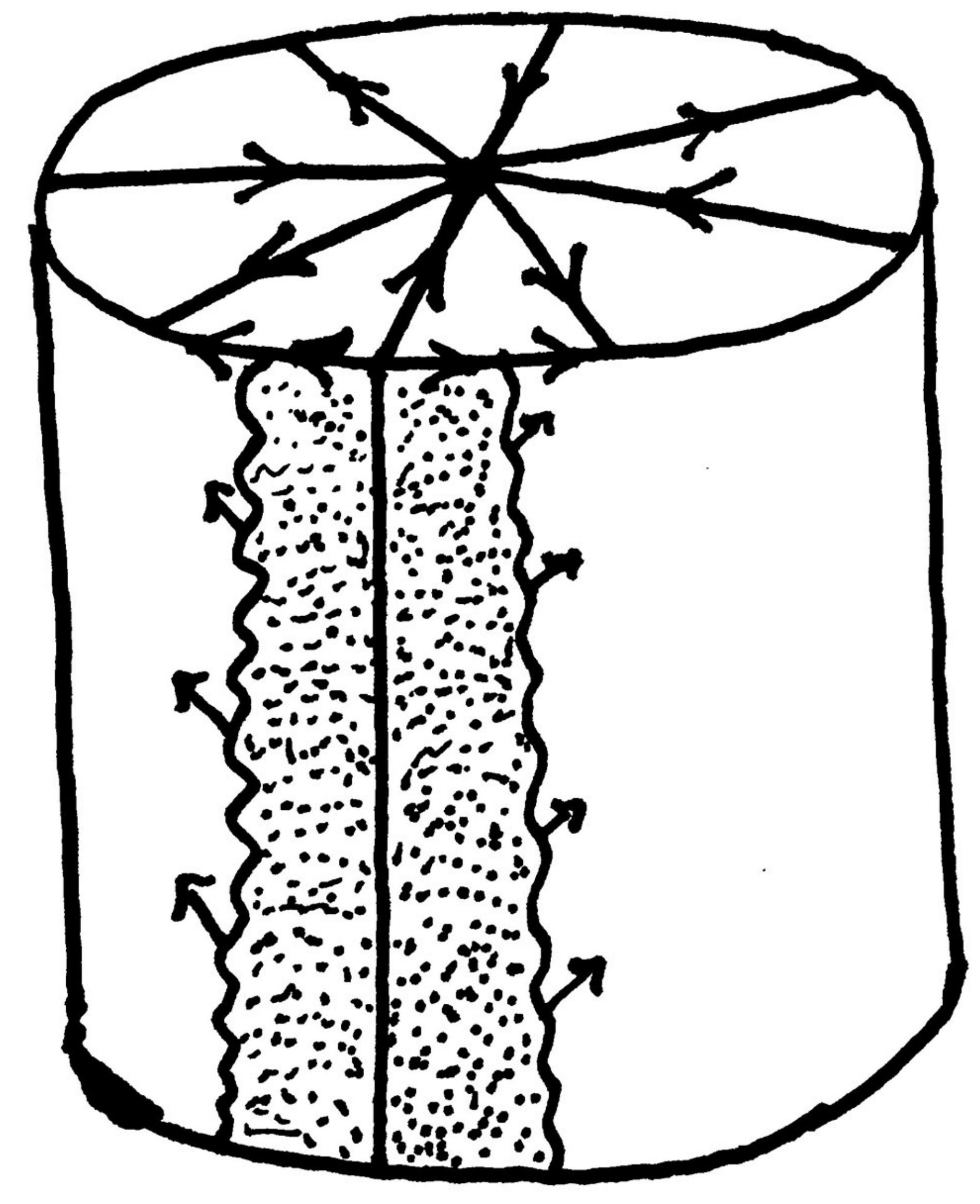}
\caption{A portion of $U_i$ is shown, illustrating implications of Lemma \ref{ladderseparatrixstructure}. The top disk is $\Delta$. The shaded region denotes an upward ladder. Its boundary ladderpoles inherit a coorientation from $B_{\mr\tau}$ agreeing with the orientations of the intervals of $\partial\Delta\setminus S$. The vertical line inside the ladder is the intersection of a stable flow prong with $\partial U_i$. }
\label{TSTproofevenfamily}
\end{figure}

 \begin{figure}[h]
\centering
\includegraphics[width=1.5in]{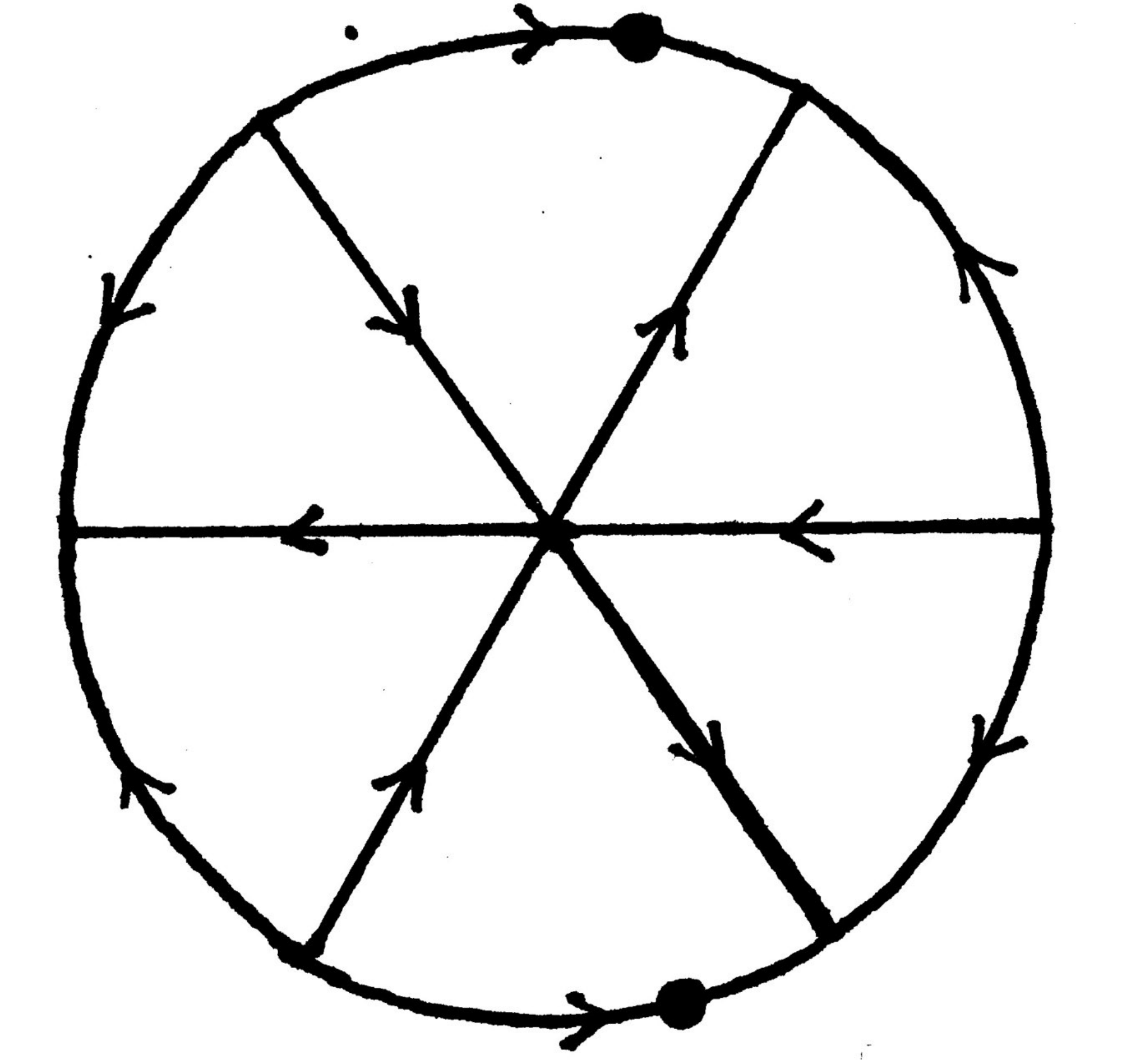}
\includegraphics[width=1.325in]{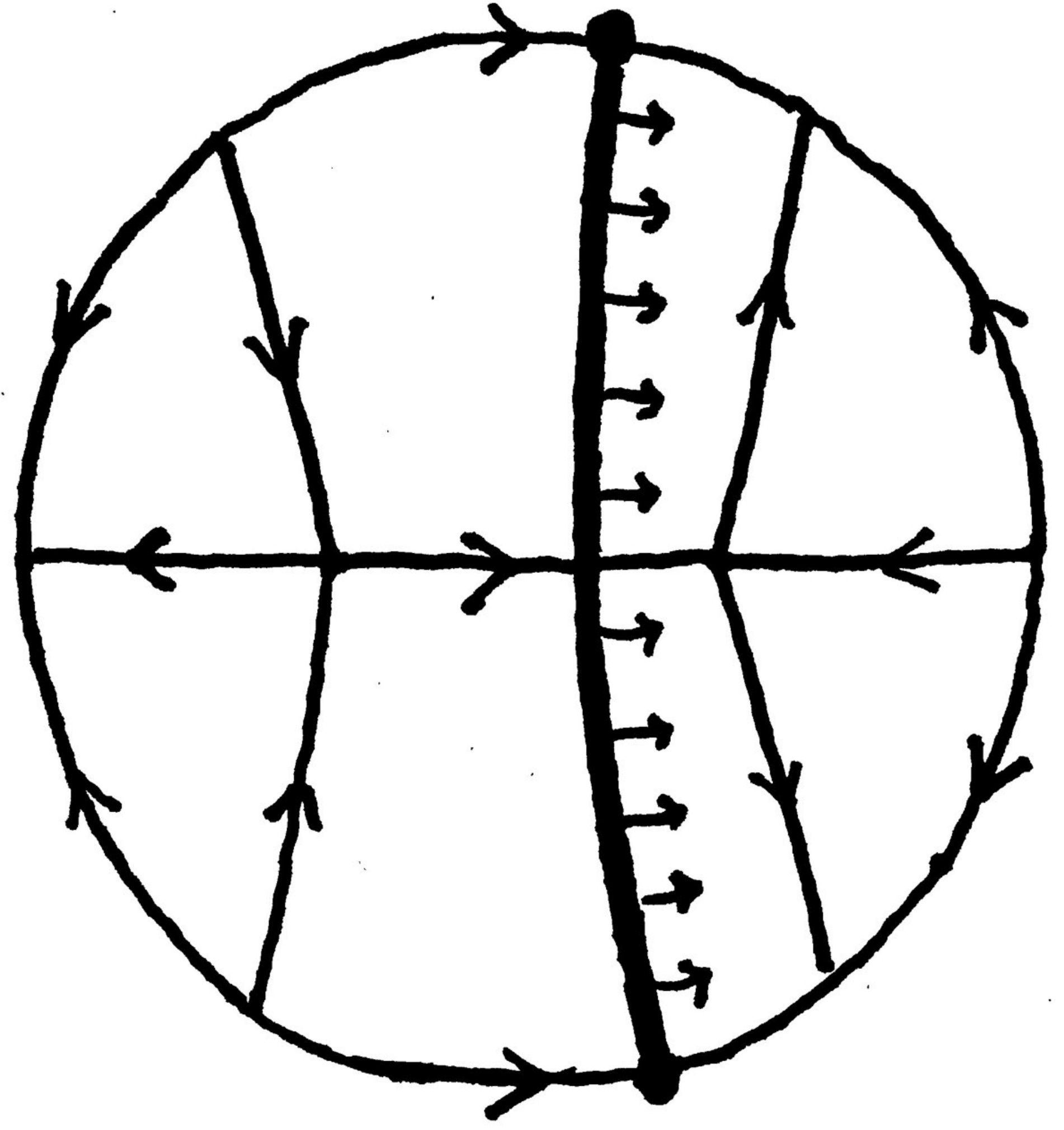}
\includegraphics[width=1.5in]{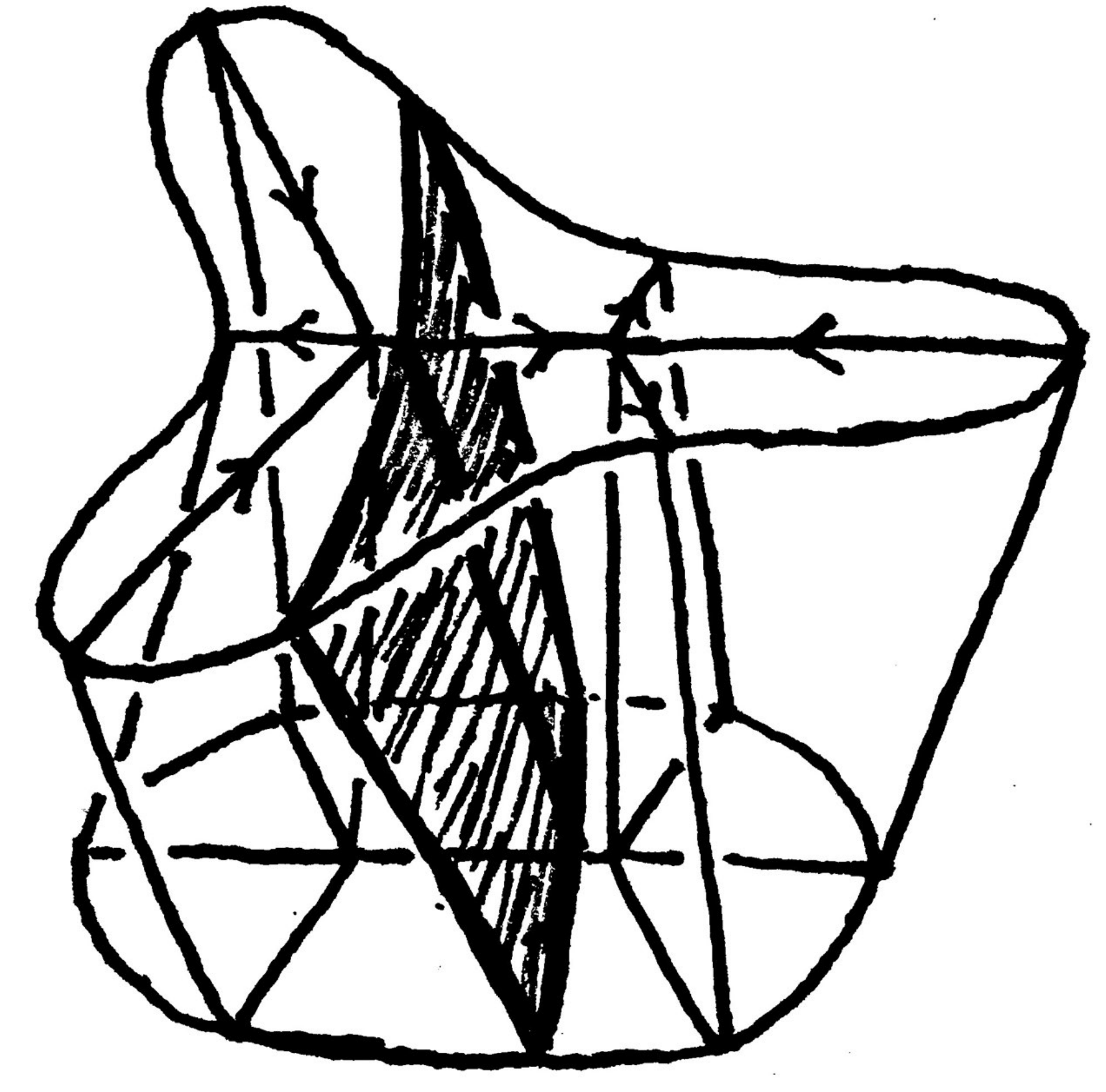}
\caption{In this example $c_i$ is a 3-pronged singular orbit, $\theta=0$, and $\mr A$ intersects $\partial U_i$ in two curves which gives an even family $E$ of size 2 (left). We dynamically blow up $S$ to $S^\sharp$ and fill in $E$ over $S^\sharp$ by $L$, which in this case is a single cooriented line segment (center). We can model $\phi^\sharp|_{U_i}$ as the vertical flow restricted to the distorted cyclinder shown on the right, with the top and bottom identified by a homeomorphism. Then we see $L$ gives rise to an annulus which is transverse to $\phi^\sharp$.}
\label{evenfamilytoannulus}
\end{figure}

\begin{proof}
By the discussion following the statement of the Transverse Surface Theorem in Section \ref{TSTstatement}, the homology class of any surface almost transverse to $\phi$ lies in $\cone(\sigma)$. Hence we need only produce a representative of $\alpha$ transverse to some dynamic blowup of $\phi$. Since any integral class in $\intr(\cone(\sigma))$ is represented by a cross section, we assume $\alpha\in\partial\cone(\sigma)$.

Our strategy will be to take a nice representative of $\mr\alpha$ which is transverse to $\phi$ and complete it over $U$ by gluing in disks and annuli which are also transverse to $\phi$. Where necessary we dynamically blow up some singular orbits of $\phi$ and glue in annuli which are transverse to the blown up flow.

By Corollary \ref{VTcarriesface}, $\mr\alpha$ has some representative $\mr A$ which is carried by $B_{\mr\tau}$, so we can assume $\mr A$ lies in $N(B_{\mr\tau})$ transverse to the normal foliation. Since $B_{\mr\tau}$ is transverse to $\phi$, so is $\mr A$. 

First, to each boundary component of $\mr A$ lying in $\partial V$ (see Section \ref{sec:notation} for notation), we glue an annulus which extends that component of $\partial\mr A$ to $\partial M$ maintaining transversality to $\phi$.

Next, let $c_i$ be a singular orbit of $\phi$ whose flow prongs $\phi$ rotates by $\theta$.

By \cite{Lan18}, $\mr A\cap \partial U_i$ is either

\begin{enumerate}[label=(\alph*)]
\item empty,
\item a collection of meridians of $\closure(U_i)$, or
\item a collection of ladderpole curves which is nulhomologous in $H_1(\partial(U_i))$.
\end{enumerate}

In case (a) we have nothing to do. In case (b) the curves can be capped off by meridional disks of $\closure(U_i)$ transverse to $\phi$. 

In case (c), we consider a meridional disk $\Delta$ of $\closure(U_i)$ which is transverse to $\phi$. The intersection of the flow prongs of $c_i$ with $\Delta$ gives a pseudo-Anosov star $S$. Further, we claim $\Delta\cap\mr A$ is an even family for $S$.

By Lemma \ref{ladderseparatrixstructure}, each interval component of $\partial\Delta\setminus S$ intersects a single ladderpole, and the orientation of the interval agrees with the coorientation the ladderpole inherits from $B_{\mr\tau}$. See Figure \ref{TSTproofevenfamily}.

Because $\partial\mr A\cap \partial U_i$ is a collection of ladderpole curves nulhomologous in $H_1(\partial U_i)$, it consists of equal numbers of left and right ladderpole curves of upward ladders. It follows that $E=\partial\mr A\cap\Delta$ is an even family for $S$.

 By Lemma \ref{blowupfill} there exists a dynamic blowup $S^\sharp$ of $S$ such that $E$ can be $\theta$-symmetrically filled in by a collection of cooriented line segments $L$ over $S^\sharp$.  The tree $S^\sharp$ determines a dynamic blowup $\phi_i^\sharp$ of $\phi$. We can suspend $L$ to a family of annuli with boundary $\partial \mr A$ that are transverse to $\phi^\sharp$ (see Figure \ref{evenfamilytoannulus} and caption).

By gluing these annuli to $\mr A$, we eliminate all boundary components of $\mr A$ meeting $\partial U_i$. The coorientations agree along $\partial\mr A$ by construction.

By repeating this procedure at every singular orbit of $\phi$, we obtain a surface $A$ which is transverse to a dynamic blowup of $\phi$. The image of $[A]$ under the puncturing map $P$ is evidently $\mr\alpha$. Since $P$ is injective, $[A]=\alpha$.
\end{proof}

\section{Homology directions and $\tau$}

Our notation for this section is the same as defined at the beginning of Section \ref{sec:notation}: $\phi$ is a circular pseudo-Anosov flow on a compact 3-manifold $M$, and $\sigma$ is the associated fibered face with veering triangulation $\tau$. By $M'$ we mean $\intr(M)\setminus \{\text{singular orbits of $\phi$}\}$.

\subsection{Convex polyhedral cones}

We recall some facts about convex polyhedral cones (for a reference see e.g. \cite{Ful93} $\S$1.2). Let $A$ be a subset of a finite-dimensional real vector space $V$. Define the \textbf{dual} of $A$ to be
\[
A^\vee=\{u\in V^*\mid u(a)\ge 0 \,\, \forall  a\in A\}.
\]
 A \textbf{convex polyhedral cone} in $V$ is the collection of all linear combinations, with nonnegative coefficients, of finitely many vectors. If $C$ is a convex polyhedral cone in $V$, then $C^\vee$ is a convex polyhedral cone in $V^*$. We have the relation $C^{\vee\vee}=C$.
 
 A \textbf{face} of $C$ is defined to be the intersection of $C$ with the kernel of an element in $C^\vee$. The \textbf{dimension} of a face of a convex polyhedral cone is the dimension of the vector subspace generated by points in the face. A top-dimensional proper face of a convex polyhedral cone is called a \textbf{facet} of the cone. If $F$ is a face of $C$, define 
\[
F^*=\{u\in C^\vee\mid u(v)=0 \,\, \forall  v\in F\}.
\]
Then $F^*$ is a face of $C^\vee$ with $\dimension(F^*)=\dimension(V)-\dimension(F)$, and $F\mapsto F^*$ defines a bijection between the faces of $C$ and the faces of $C^\vee$. We indulge in some foreshadowing by remarking that, in particular, $F\mapsto F^*$ restricts to a bijection between the one-dimensional faces of $C$ and the facets of $C^\vee$.

We can identify $H_1(M)$ with $H^1(M)^*$ via the universal coefficients theorem, so if $C$ is a convex polyhedral cone in $H^1(M)$ we will view $C^\vee$ as living in $H_1(M)$ and vice versa.

\subsection{Flipping}

Recall that $B_{\mr\tau}$ is the the 2-skeleton of $\mr\tau$ viewed as a branched surface.
Let $\tri$ be a truncated taut tetrahedron of $\mr \tau$. If $S$ is carried with positive weights on both of the sectors of $B_{\mr\tau}$ corresponding to the bottom of $\tri$, then $S$ may be isotoped upwards through $\tri$ to a new surface carried by $B_{\mr\tau}$ such that the uppermost (with respect to the orientation of the normal foliation of $N(B_{\mr\tau})$) portion of $S$ which was carried by the bottom of $\tri$ is now carried by the top of $\tri$. Outside of a neighborhood of $\tri$ this isotopy is the identity. We call this isotopy an \textbf{upward flip}. If $S_1$ is the image of $S$ under a single upward flip of $S$, we say $S_1$ is an \textbf{upward flip of $S$}.

\subsection{A train track on $\tau$}\label{sec:traintrack}

Let $T$ be a traintrack embedded in the 2-skeleton of $\tau$ with a single trivalent switch lying in the interior of each ideal triangle such that each edge of $\tau$ intersects $T$ in a single point. Since each edge of $\tau$ has degree at least 4, this point is a switch of $T$ with valence $\ge 4$. Note that $T$ is nonstandard in 2 ways: it is embedded in the 2-skeleton of $\tau$ rather than a surface, and it is not trivalent.

Let $t$ be a triangle of $\tau$, and let $s$ be the switch of $T$ interior to $t$. The interior of $t$ is divided into 3 disks by $T\cap t$, one of which has a cusp at $s$. 
The branch which is disjoint from the cusped region is called a \textbf{large branch} of $T$. A branch which is not large is called a \textbf{small branch} of $T$.

We now define a particular train track in the 2-skeleton of $\tau$ which we call the \textbf{stable train track of $\tau$} and denote by $\St(\tau$).
For each triangle $t$ of $\tau$ we place a trivalent switch in the interior of $t$, and connect the large branch to the unique edge of $t$ which is the bottom $\pi$-edge of the taut tetrahedron which $t$ bounds below. We connect the two small branches to the other two edges of $t$. Gluing so that $\St(\tau)$ intersects each edge of $\tau$ in a point yields our desired train track. 
This is precisely the train track we would get from Agol's construction of $\tau$ by fixing a fibration, building $\tau$ as a layered triangulation on the fiber by looking at dual triangulations to a periodic maximal splitting sequence of the monodromy's stable train track, and recording the switches of stable train tracks on the relevant ideal triangles.
We will also view $\St(\tau)$ as living in $B_{\mr\tau}$.
Figure \ref{visualizingthetraintrack} shows a portion of a veering triangulation with its stable train track.

Any surface $S$ carried by $B_{\mr\tau}$ naturally inherits a trivalent train track from $\St(\tau)$ which we call $\St(S)$. There is a natural cellular map $\St(S)\to \St(\tau)$, which allows us to identify each branch of $\St(S)$ with two branches of $\St(\tau)$. This map need be neither surjective nor injective. A branch of $\St(S)$ composed of 2 large branches of $\St(\tau)$ is called a \textbf{large branch} of $\St(S)$.

 By construction, we have the following.

\begin{observation}
Let $S$ be a surface carried by $B_{\mr\tau}$. There exists an upward flip of $S$ if and only if $\St(S)$ contains a large branch.
\end{observation}

Any curve carried by $\St(S)$ corresponds to a curve carried by $\St(\tau)$ under the map $\St(S)\to \St(\tau)$, and we will abuse terminology slightly by considering a curve carried by $\St(S)$ to also be carried by $\St(\tau)$.

\subsection{Infinite flippability} \label{sec:flippability}

A finite or infinite sequence $\{S_i\}$ of surfaces carried by $B_{\mr\tau}$ such that each successive element is an upward flip of the previous element is called a \textbf{flipping sequence}. Any element of an infinite flipping sequence is called \textbf{infinitely flippable}.

If $A$ is a surface carried by a branched surface with positive weights on each sector, we say $[A]$ is \textbf{fully carried}.

We use these new words in a mathematical sentence:

\begin{observation}\label{fiber implies infinite flippability}
Let $\mr A$ be a surface carried by $B_{\mr\tau}$ which is a fiber of $\mr M$. Because $\tau$ can be built as a layered triangulation on the extension of $\mr A$ to $M'$, $\mr A$ is infinitely flippable and some positive integer multiple of $[\mr A]$ is fully carried by $B_{\mr \tau}$.
\end{observation}

%A $\tau$-transversal which does not visit the same tetrahedron twice will be called \textbf{minimal}.

The cone of homology directions of a flow $F$, denoted $\C_F$, is the smallest closed cone containing the projective accumulation points of nearly closed orbits of $F$. Since in our case $\phi$ is a circular pseudo-Anosov flow, there is a more convenient characterization of $M$ as the smallest closed, convex cone containing the homology classes of the closed orbits of $\phi$ (see the proof of Lemma \ref{it's a nice cone} in Appendix \ref{Fried appendix}). In fact, it suffices to take the smallest convex cone containing a certain \emph{finite} collection of closed orbits. It follows that $\C_\phi$ is a rational convex polyhedral cone.

Let $\tau$ be the veering triangulation of $g$. Define a \textbf{$\tau$-transversal} to be an oriented curve in $M'$ which is \textbf{positively transverse} to the 2-skeleton of $\tau$, i.e. intersects the 2-skeleton only transversely and agreeing with the coorientation of $\tau$.
Let $\C_\tau\subset H_1(M)$ be the smallest closed cone containing the homology class of each closed $\tau$-transversal.

\begin{proposition}\label{a cone is a cone}
Let $\ell$ be a closed $\tau$-transversal. Then $\ell\in \C_\phi\setminus\{0\}$. Moreover, $\C_\phi=\C_\tau$.
\end{proposition}

\begin{proof}
Since both cones are closed, to show $\C_\phi=\C_\tau$ it suffices to show that the homology class of every closed orbit lies in $\C_\tau$, and that the homology class of every closed $\tau$-transversal lies in $\C_\phi$.

Suppose $o$ is a closed orbit of $\phi$. If $o$ lies interior to $\intr(M)\setminus c$, then $o$ is already a closed $\tau$-transversal. Otherwise $o$ is a singular or $\partial$-singular orbit and can be isotoped onto a closed $\tau$-transversal lying in the interior of a ladder in $\partial \mr M$. Hence $\C_\phi\subset \C_\tau$.

Now suppose $\ell$ is a closed $\tau$-transversal. We can isotope $\ell$ into $\mr M$ such that $\ell$ is positively transverse to $B_{\mr\tau}$.
Let $\beta\in\intr(\cone(\sigma))$ be an integral class, and let $\mr B$ be a representative of $\mr \beta$ carried by $B_{\mr\tau}$. Since $\mr B$ is a fiber of $\mr M$, by Observation \ref{fiber implies infinite flippability} there exists a surface $n\mr B$ (topologically $n\mr B$ is $n$ parallel copies of $\mr B$ for some positive integer $n$) representing $n\mr\beta$ which is fully carried by $B_{\mr\tau}$. We can cap off the boundary components of $n\mr B$ to obtain a surface $nB$ in $M$ representing $n\beta$ whose intersection with $\mr M$ is $n\mr B$. Since $\ell$ is positively transverse to $B_{\mr\tau}$, it has positive intersection with $nB$. Letting $\beta_{\LD}$ denote the Lefschetz dual of $\beta$, we see $\beta_{\LD}([\ell])>0$. Viewing $[\ell]$ as a linear functional on $H^1(M)$, we see $[\ell]$ is strictly positive on $\intr(\sigma_{\LD})$, so $[\ell]$ is nonnegative on $\cone(\sigma_{\LD})$, whence $\ell\in \cone(\sigma_{\LD})^\vee\setminus\{0\}$.

Note that by Theorem \ref{generalized Fried} we have $\cone(\sigma_{\LD})=(\C_\phi)^\vee$. Hence $\cone(\sigma_{\LD})^\vee=\C_\phi$, so $[\ell]\in \C_\phi\setminus\{0\}$. \end{proof}

\begin{proposition}\label{flipping fibers}
A surface carried by $B_{\mr\tau}$ is a fiber of $\mr M$ if and only if it is infinitely flippable.
\end{proposition}

\begin{proof}
One direction of this is just Observation \ref{fiber implies infinite flippability}.

For the other direction, we will show that if a flipping sequence is such that there is a 3-cell of $\mr\tau$ which is not swept across by a flip in the sequence, then the sequence is finite. Therefore if $S$ is infinitely flippable, some integer multiple of $[S]$ will be fully carried by $B_{\mr\tau}$. Any fully carried homology class has positive intersection with any closed transversal to $\tau$ and thus represents a fiber by Proposition \ref{a cone is a cone} and Theorem \ref{generalized Fried}.

Let $\{S_i\}$ be a flipping sequence. We assume $S_1$ is connected, as otherwise we can apply the following reasoning to the flipping sequence associated to each component of $S_1$. Let $R$ be the union of all 2-cells in every $S_i$, as well as every 3-cell swept across by an upward flip in the sequence. Suppose $R\subsetneq M$. We claim there is some 3-cell of $\mr\tau$ whose bottom $\pi$-edge lies in $R$ and which is not swept across by an upward flip in our sequence.

There is some 3-cell $t\not\subset R$ incident to $R$ along one of its edges. This edge meets two torus boundary components of $\mr M$; pick one and call it $T$. Then $T\cap R$ is a nonempty proper subset of $T$ which is simplicial with respect to the triangulation of $T$ by flat triangles coming from $\mr \tau$. To find a 3-cell not swept across by an upward flip whose bottom $\pi$-edge lies in $R$, it suffices to find a flat triangle $\tri$ in $T$ such that $\tri\not\subset T\cap R$ and such that $\tri$ has an edge lying in $T\cap R$ whose coorientation points into $\tri$. This is equivalent to finding an edge of $T\cap R$ lying in $\partial(T\cap R)$ whose coorientation points into $T\setminus(T\cap R)$. We will call such an edge an \textbf{outward pointing edge} of $T\cap R$.

To find an outward pointing edge of $T\cap R$, we use our knowledge of the combinatorics of the triangulation of $T$. If the intersection of $T\cap R$ with the interior of some ladder is a nonempty proper subset of the ladder, we can find a rung of that ladder which is an outward pointing edge of $T\cap R$. Otherwise, $T\setminus(T\cap R)$ contains the interior of at least one ladder. As was observed in \cite{Lan18}, any closed curve carried by $\partial \mr \tau\cap T$ which is not a ladderpole curve must traverse each ladder of $T$. We conclude that $\partial S_i\cap T$ must be a collection of ladderpole curves for all $i$. Any two edges of a ladderpole which share a vertex cannot correspond to a pair of 2-cells of $\mr\tau$ forming the bottom of a 3-cell, because each bottom of a 3-cell intersects $T$ in at least one rung. Therefore there are no flips of $S_i$ incident to $T$ for any $i$ and we conclude that $R\cap T$ is a collection of ladderpole curves, so every edge of $R\cap T$ is outward pointing.

Hence, as claimed, we can produce a 3-cell $t\not\subset R$ of $\mr\tau$ whose bottom $\pi$-edge $e$ lies in $R$, and thus in $S_k$ for some $k$. (We have actually produced $t$ with one of its bottom faces lying in $R$).

%Let $S$ be an infinitely flippable surface carried by $B_{\mr\tau}$, with infinite flipping sequence $S=S_1,S_2,S_3,\dots$. Suppose for a contradiction that there exists some 3-cell $t$ of $\mr\tau$ that is not swept across by an isotopy from $S_i$ to $S_{i+1}$ for any $i$. We may choose $t$ so that its bottom $\pi$-edge lies in $S_k$ for some $k$. 

Any surface carried by $B_{\mr\tau}$ inherits an ideal triangulation of its interior, and a flipping sequence gives a sequence of diagonal exchanges between ideal triangulations of a reference copy of the surface. Let $\Sigma$ be a reference copy of $S_k$. Then $\{S_i\}_{i\ge k}$ gives a sequence of ideal triangulations of $\Sigma$ related by diagonal exchanges, and $e$ is an edge of each triangulation. 

We say an edge in this sequence of diagonal exchanges is \textbf{adjacent } to a particular diagonal exchange if it is a boundary edge of the ideal quadrilateral whose diagonal is exchanged.

Because each edge of $\mr \tau$ is incident to only finitely many tetrahedra, each edge in this sequence of ideal triangulations of $\sigma$ can be adjacent to only finitely many diagonal exchanges before it either disappears or remains forever. Since $e$ is present in each triangulation, there exists $j\ge k$ and edges $e', e''$ which form a triangle with $e$ such that the triangle $(e,e',e'')$ is present in $S_i$ for $i\ge j$. Since $e'$  is adjacent to only finitely many diagonal exchanges, it is also eventually incident to a triangle $(e',e''',e'''')$ that is fixed by the sequence of diagonal exchanges, and similarly for $e''$. Each triangulation has the same number of triangles, so continuing in this way we eventually cover $\Sigma$ by triangles which are fixed. Therefore the sequence is finite.
\end{proof}

\subsection{Stable loops}\label{sec:stableloops}

Let $\lambda$ be a closed curve carried by $\St(\tau)$. If $\lambda$ has the property that it traverses alternately small and large branches of $\St(\tau)$, we call $\lambda$ a \textbf{stable loop}. If $\lambda$ additionally has the property that it traverses each switch of $\St(\tau)$ at most once, then we say $\lambda$ is a \textbf{minimal stable loop}. Since $\tau$ consists of finitely many ideal tetrahedra, $\St(\tau)$ has finitely many switches and thus finitely many minimal stable loops.
We endow each stable loop $\lambda$ with an orientation such that at a switch in the interior of a 2-cell, $\lambda$ passes from a large branch to a small branch (see Figure \ref{lazyriver}).

Note that for any veering triangulation $\rho$, $\St(\rho)$ has stable loops. It is easily checkable that for any 2-cell $\tri$ of $\mr\rho$, $\St(\rho)\cap \tri$ determines the left or right veeringness of two out of the three edges of $\tri$ not lying in the boundary of the ambient 3-manifold, as shown in Figure \ref{veeringrule}. 
\begin{figure}[h]
\centering
\includegraphics[height=1in]{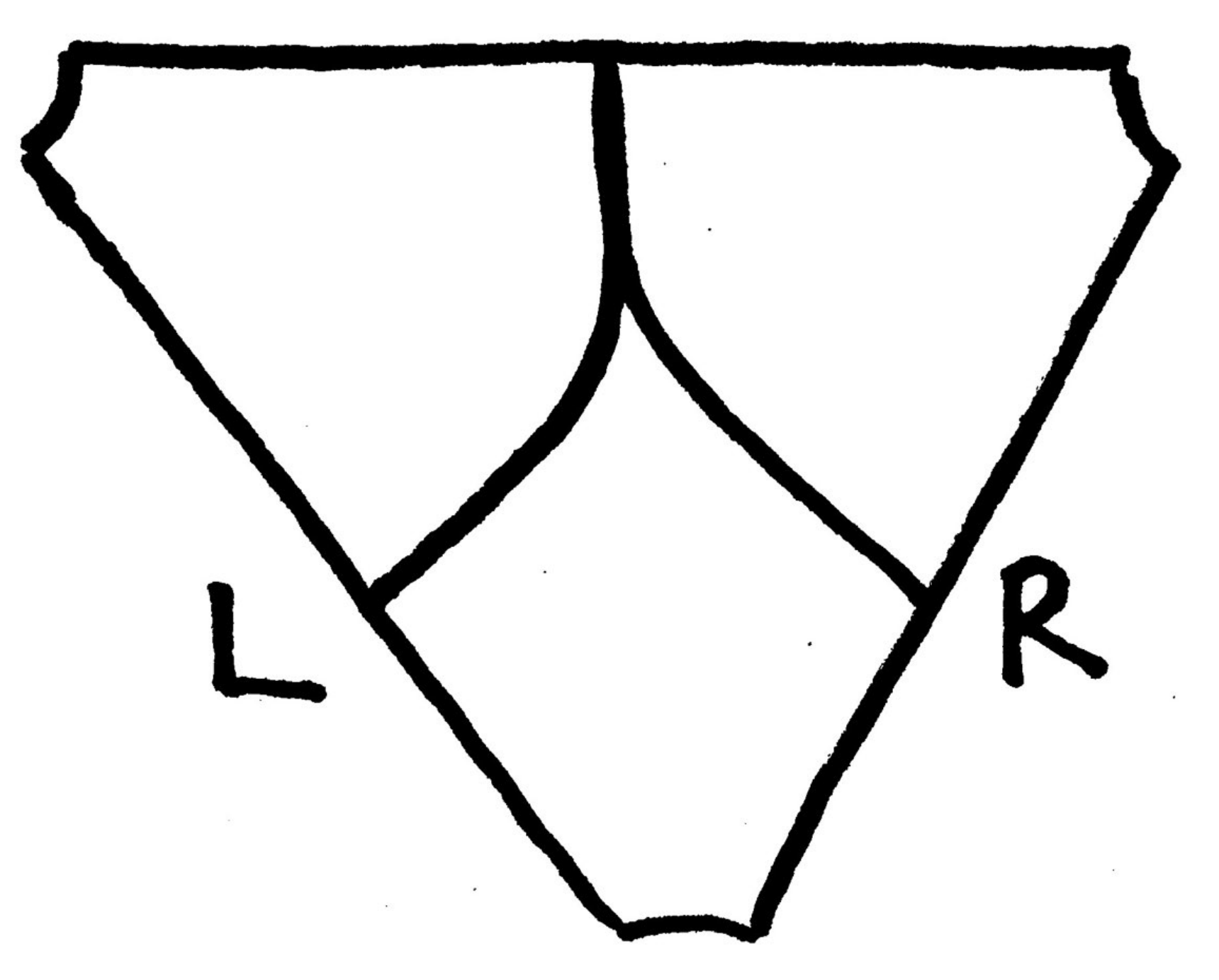}
\caption{The intersection of the stable traintrack of a veering triangulation restricted to any 2-cell determines the left or right veeringness of two edges as shown, where the coorientation is pointing out of the page.}
\label{veeringrule}
\end{figure}
The small branch of $\St(\rho)$ incident to a right veering edge is called a \textbf{right small branch}, and a \textbf{left small branch} is defined symmetrically.
To produce a stable loop we can choose, for example, the left ladderpole $L$ of an upward ladder and look at the collection $\Lambda$ of all 2-cells of $\mr\rho$ meeting $L$. The edges of $\mr\rho$ corresponding to vertices in $L$ are all right veering by Lemma \ref{lem:veeringrule}, so $\St(\rho)\cap \Lambda$ is as shown in Figure \ref{ladderpoleloop}, 
\begin{figure}[h]
\centering
\includegraphics[height=2in]{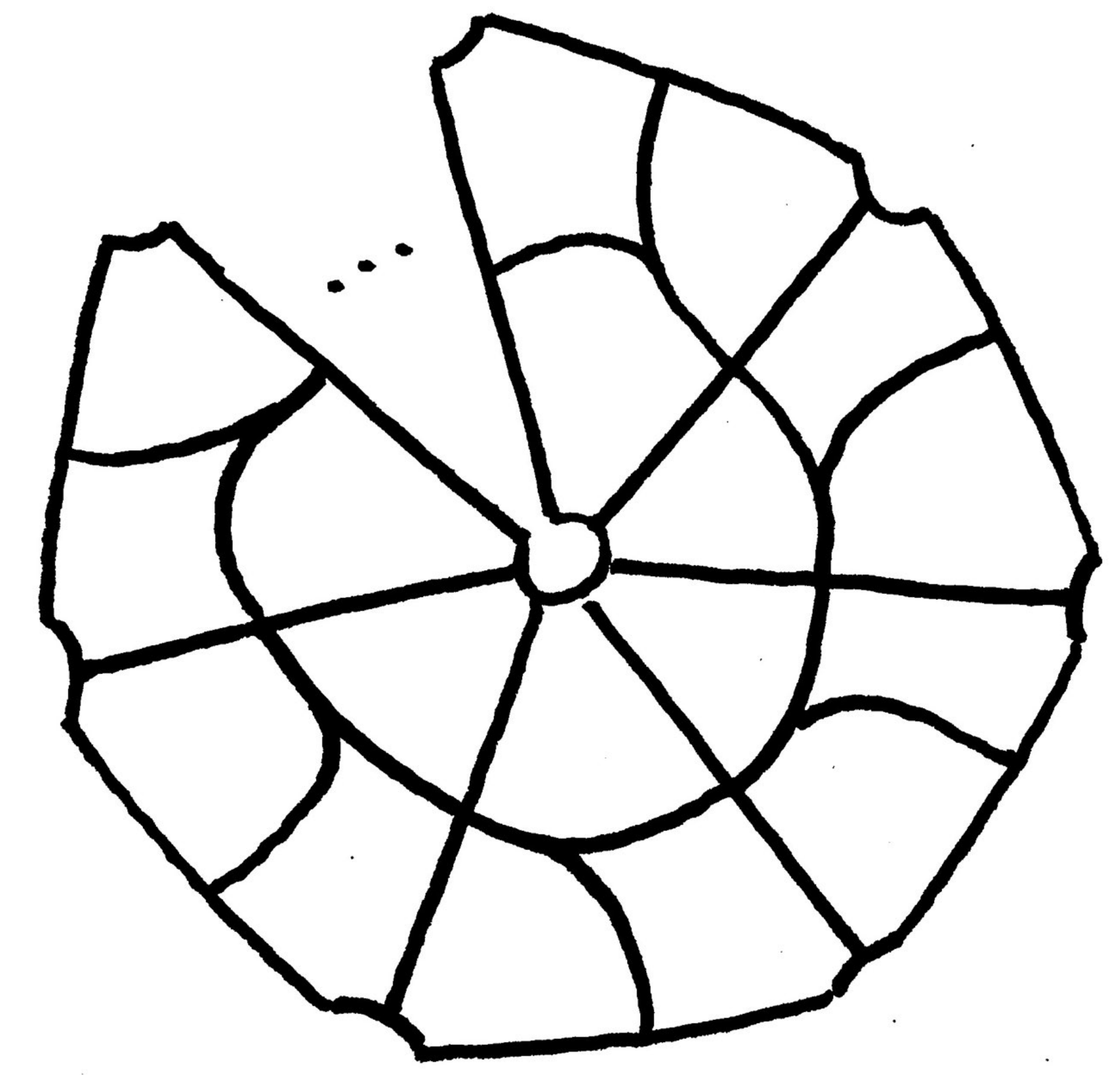}
\caption{The 2-cells incident to a ladderpole give rise to a stable loop.}
\label{ladderpoleloop}
\end{figure}
and carries a stable loop. We claim this stable loop is also minimal, or equivalently that the edges of $L$ are in bijection with the 2-cells of $\Lambda$. This is true: for each 2-cell $\tri$ meeting $L$, $\tri$ is incident to $L$ along the \emph{unique} truncated ideal vertex bounded by the edges meeting the large branch and right small branch of $\St(\rho)\cap \tri$.

Observe that for any stable loop $\lambda$ carried by $\St(\tau)$, we have $[\lambda]\in \C_\tau$. This is because the condition that $\lambda$ traverses alternately large and small branches of $\St(\tau)$ implies that $\lambda$ can be perturbed to a closed $\tau$-transversal, as shown in Figure \ref{lazyriver3d}.
Hence we have the following Lemma.

\begin{lemma}\label{looptodirection}
Let $\lambda$ be a stable loop. Then $[\lambda]\in \C_\phi$.
\end{lemma}

\begin{figure}[h]
\centering
\includegraphics[height=1.5in]{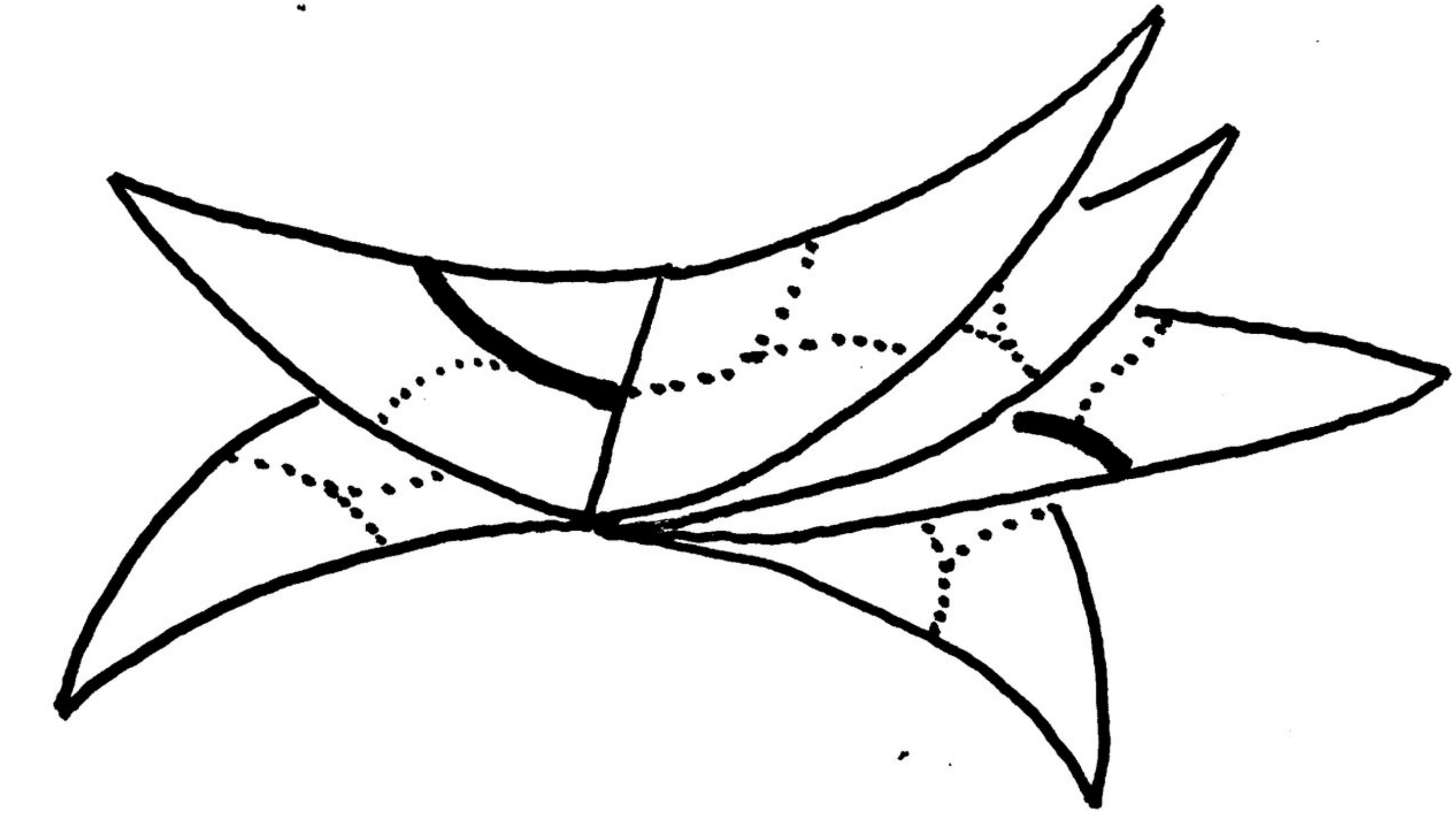}
\includegraphics[height=1.5in]{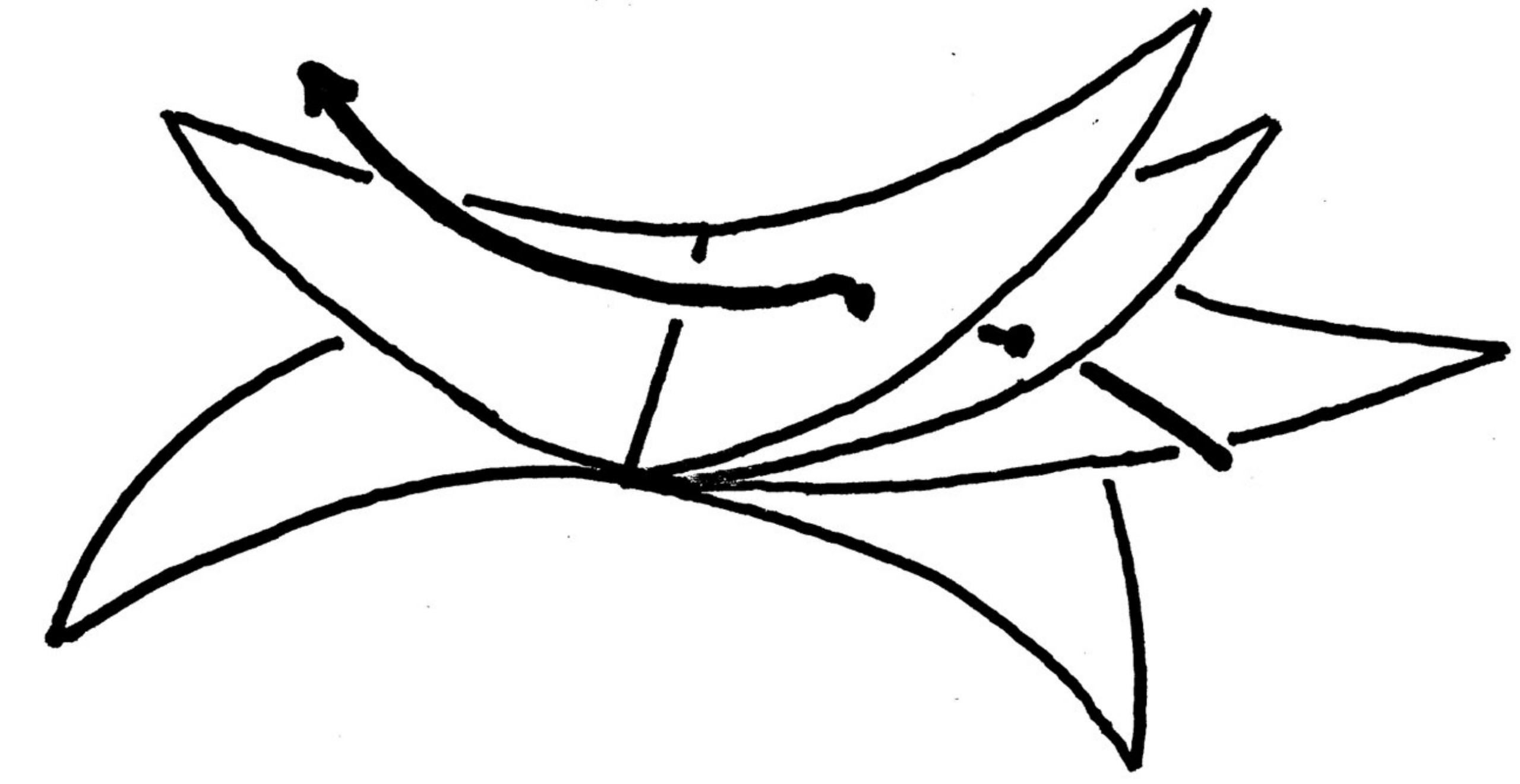}
\caption{A stable loop may be perturbed to a closed $\tau$-transversal.}
\label{lazyriver3d}
\end{figure}

Let $S$ be a surface carried by $\mr \tau$ which is not a fiber of a fibration $\mr M\to S^1$. By Theorem \ref{flipping fibers}, any flipping sequence starting with $S$ is finite. Therefore $S$ is isotopic to a surface $S'$ which carried by $B_{\mr\tau}$ such that $S'$ has no upward flips. We call such a surface \textbf{unflippable}.

\begin{proposition}\label{unflippabletolazy}
Let $S$ be a surface carried by $B_{\mr\tau}$ which is unflippable. Then $\St(S)$ carries a stable loop of $\tau$.
\end{proposition}

\begin{proof}
The unflippability of $S$ is equivalent to $\St(S)$ having no large branches. We define a curve carried by $\St(S)$ as follows: start at any switch of $\St(S)$, and travel along its large half-branch. When arriving at the next switch, exit along that switch's large half-branch, and so on. Since $\St(S)$ has finitely many branches, eventually the path will return to a branch it has previously visited, at which point we obtain a closed curve carried by $\St(S)$. By construction it alternates between large and small branches of $\St(\tau)$. 
\end{proof}

 We now prove the paper's main result. Recall that the subscript $LD$ attached to an object denotes its image under Lefschetz duality. 

\begin{theorem}[Stable loops]\label{stable loop theorem}
Let $M$ be a compact hyperbolic 3-manifold with fibered face $\sigma$. Let $\tau$ and $\phi$ be the associated veering triangulation and circular pseudo-Anosov flow, respectively. Then $\C_\phi$ is the smallest convex cone containing the homology classes of the minimal stable loops of $\tau$.
\end{theorem}

\begin{proof}
By Lemma \ref{looptodirection}, the cone generated by the minimal stable loops lies in $\C_\phi$.
Hence it suffices to show that every 1-dimensional face of $\C_\phi$ is generated by the homology class of a minimal stable loop.

Suppose first that $\dimension(H_2(M,\partial M))=1$. Then $H_1(M)=\R$, and $\C_\phi$ is a ray. Let $\lambda$ be a minimal stable loop of $\tau$. Since $[\lambda]\in \C_\phi\setminus\{0\}$ by Proposition \ref{a cone is a cone}, we have proved the claim in this case.

Now suppose that $\dimension(H_2(M,\partial M))>1$. Consider a 1-dimensional face $\Phi$ of $\C_\phi=(\cone(\sigma_{\LD}))^\vee$, which can be characterized as $\Phi=(F_{\LD})^*$ for some facet $F$ of $\cone(\sigma)$.

Let $\alpha\in H_2(M,\partial M)$ be a primitive integral class lying in the relative interior of $F$. By Corollary \ref{VTcarriesface}, Proposition \ref{flipping fibers}, and Proposition \ref{unflippabletolazy}, there exists an unflippable surface $\mr S$ representing $\mr \alpha$ and carried by $B_{\mr\tau}$ such that $\St(\mr S)$ carries a stable loop $\lambda$.

As in the proof of Proposition \ref{a cone is a cone}, we can cap off $\mr S$ to obtain a surface $S$ in $M$ representing $\alpha$ with $S\cap\mr M=\mr S$. Since $\lambda$ may be isotoped off of $S$, $\alpha_{\LD}\in \ker([\lambda])$, viewing $[\lambda]$ as a linear functional on $H^1(M)$. It follows that $[\lambda]\in \Phi\setminus\{0\}$, so $\Phi$ is generated by $[\lambda]$. Because $F$ has codimension 1, $(F_{\LD})^*$ has dimension 1 and is thus generated by $[\lambda]$.

Since $\mr S$ is embedded in $N(B_{\mr\tau}$, $\lambda$ never traverses a switch of $\St(\tau)$ in two directions. This is clear for switches of $\St(\tau)$ lying in the interiors of 2-cells, and for a switch lying on an edge of $\mr \tau$, such behavior would force $\mr S$ to be non-embedded. 

Any curve carried by a train track which traverses each switch in at most one direction, and also never traverses the same branch twice, must never traverse the same vertex twice. 
Therefore by cutting and pasting $\lambda$ along any edges of $St(\tau)$ carrying $\lambda$ with weight $>1$, we see $\lambda$ is homologous to a union of minimal stable loops $\lambda_1,\dots,\lambda_n$:
\[
[\lambda]=[\lambda_1]+\cdots+[\lambda_n].
\]
Since $[\lambda]$ lies in a one-dimensional face of $\C_\phi$, we conclude that for each $n$ there is some positive integer $n_i$ such that $\lambda=n_i[\lambda_i]$. It follows that $\Phi$ is generated by the homology class of a minimal stable loop.
\end{proof}

\begin{remark}
We could just as easily have defined \textbf{unstable loops} using a symmetrically defined \textbf{unstable train track} of $\tau$, and flipped surfaces downward in the arguments above to achieve the corresponding results.
\end{remark}

Combining Proposition \ref{a cone is a cone} and Theorem \ref{stable loop theorem} we have the following immediate corollary, which is not obvious from the definitions.

\begin{corollary}
The cone $\C_\tau$ is the smallest convex cone containing the homology classes of the stable loops of $\tau$.
\end{corollary}

%In summary, we have defined 3 cones: the convex cone generated by the stable loops of $\tau$, which we can denote by $\C_{St}$, as well as $\C_\phi$ and $\C_\tau$. In the case we have been considering, when $\tau$ is layered, these 3 cones are equal. However, given a \emph{nonlayered} veering triangulation and a transverse pseudo-Anosov flow\footnote{the forthcoming paper \cite{SchSeg19} will make clear that this situation is far from uncommon.}, a priori one has only the inclusion $\C_{St}\subset \C_\tau$. One might reasonably expect that $\C_\phi\subset \C_\tau$, but our proof of this containment in the layered case (in Proposition \ref{a cone is a cone}) used the layeredness of $\tau$.

%%%APPENDICES%%%%%%%%%%%%%%%%%%%%%%%%%%%%%%%%%%%%%%%%%%%%%%%%%%%%%%%%%%%%%%%%%%%%%%%

\begin{appendices}

\section{The suspension flow is canonical when our manifold has boundary}\label{Fried appendix}

The purpose of this appendix is to record a generalization of Fried's theory relating circular pseudo-Anosov flows on closed 3-manifolds to fibered faces that does not currently exist in the literature. The generalization here is to the case of circular pseudo-Anosov flows on compact 3-manifolds, possibly with boundary.

\begin{appthm}
Let $\phi$ be a circular pseudo-Anosov flow on a compact 3-manifold $M$ with cross section $Y$. Let $\sigma$ be the fibered face of $B_x(M)$ such that $[Y]\in\cone(\sigma)$. Let $\alpha\in H_2(M,\partial M)$ be an integral class. The following are equivalent:
\begin{enumerate}
\item $\alpha$ lies in $\intr(\cone(\sigma))$
\item the Lefschetz dual of $\alpha$ is positive on the homology directions of $\phi$
\item $\alpha$ is represented by a cross section to $\phi$.
\end{enumerate}
Moreover, $\phi$ is the unique circular pseudo-Anosov flow admitting cross sections representing classes in $\cone(\sigma)$ up to reparameterization and conjugation by homeomorphisms of $M$ isotopic to the identity.
\end{appthm}

Homology directions are essentially projectivized homology classes of nearly-closed orbits of $\phi$. Their precise definition is given below, in Section \ref{Fried background}.

Some experts may have verified for themselves that this generalization works. However, the results in this paper depend on it so we include this proof. Our proof attempts to follow the arguments of Fried, making modifications when necessary to deal with boundary components.

\subsection{Homology directions and cross sections}\label{Fried background}

We begin by recalling some definitions and a result from \cite{Fri82}. Let $M$ be a compact smooth manifold, and let $D_M$ be the quotient of $H_1(M)$ by positive scalar multiplication, endowed with the topology of the disjoint union of a sphere and an isolated point corresponding to 0. We denote the quotient map $H_1(M)\to D_M$ by $\pi$. Let $\phi$ be a $C^1$ flow on $M$ which is tangent to $\partial M$. Let $\phi_t(a)$ denote the image of $a\in M$ under the time $t$ map of $\phi$.

A \textbf{closing sequence based at $m\in M$} is a sequence of points $(m_k,t_k)\in M\times \R$ with $m_k\to m\in M$, $\phi_{t_k}(m_k)\to m$, and $t_k\nrightarrow 0$. For sufficiently large $k$, the points $m_k$ and $\phi_{t_k}(m_k)$ lie in a small ball $B$ around $M$. We can define a closed curve $\gamma_k$ based at $m$ by traveling along a short path in $B$ from $m$ to $m_k$, flowing to $\phi_{t_k}(m_k)$, and returning to $m$ by a short path in $B$. The $\gamma_k$ are well-defined up to isotopy.

Since $D_M$ is compact, $\pi([\gamma_k])$ must have accumulation points. Any such accumulation point $\delta$ is called a \textbf{homology direction for $\phi$}. We call $(m_k,t_k)$ a \textbf{closing sequence for $\delta$} if $\pi([\gamma_k])\to\delta$. The set of homology directions for $\phi$ is denoted $D_\phi$. Let $\delta\in D_M$ and $\alpha\in H^1(M)$. While $\alpha(\delta)$ is not well-defined unless we choose a norm on $H_1(M)$ and identify $D_M$ with the vectors of length 0 and 1, we can say whether $\alpha$ is positive, negative, or zero on $\delta$.

We define $\C_\phi$, the \textbf{cone of homology directions} of $\phi$, by $\C_\phi:=\pi^{-1}(D_\phi\cup\{0\})$.

Fried gives a useful criterion for when an integral cohomology class in $H^1(M)$ is \textbf{compatible with a cross section}, i.e. has a cross section of $\phi$ representing its Lefschetz dual.

\begin{theorem}[Fried]\label{cross section theorem}
Let $\alpha\in H^1(M)$ be an integral class. Then $\alpha$ is compatible with a cross section to $\phi$ if and only if $\alpha(\delta)>0$ for all $\delta\in D_\phi$.
\end{theorem}

We relate a couple of useful observations of Fried. The first is that if $(m_k,t_k)$ is a closing sequence for $\delta\in D_\phi$ based at $m$ and $t_k$ has a bounded subsequence, then $m$ lies on a periodic orbit $o_m$ and $\delta=\pi([o_m])$. Thus $(m_k=m,t_k=kp)$ where $p$ is the period of $m$ is also a closing sequence for $\delta$. The second observation is that if $\phi$ admits a cross section $Y$, then each $\delta\in D_\phi$ admits a closing sequence $(m_k,t_k)$ with $m,m_k,\phi_{t_l}(m_k)\in Y$. This can be seen by flowing each point in a closing sequence for $\delta$ until it meets $Y$ for the first time. We record these observations in a lemma.

\begin{lemma}[Fried]\label{infinitet}
Suppose $\phi$ admits a cross section $Y$, and let $\delta\in D_\phi$. Then $\delta$ admits a closing sequence $(m_k,t_k)$ based at $m$ with $m,m_k,\phi_{t_k}(m_k)\in Y$ and $t_k\to \infty$.
\end{lemma}

\subsection{Proving Theorem \ref{generalized Fried}}

Throughout this section, our notation mimics that in Section \ref{canonical veering section}. We consider a compact hyperbolic 3-manifold $M$  and  a circular pseudo-Anosov flow $\phi$ on $M$ admitting a cross section $Y$ with first return map $g$. We assume $\phi$ is parameterized so that for all $z\in Z$ we have $\phi_1(z)=g(z)$.

As in \cite{FLP79}, we can find a Markov partition $\M$ for $g\colon Y\to Y$, where the definition of Markov partition is altered slightly to account for the fact that $Y$ may have boundary (Markov ``rectangles" touching $\partial Y$ are actually pentagons). We will call the elements of $\M$  \textbf{shapes}, and the elements of $\M$ which touch $\partial Y$ \textbf{pentagons}. By the construction in \cite{FLP79} we can assume that an edge of a pentagon meeting $\partial Y$ in a single point is contained in a stable or unstable prong of the stable or unstable foliation of $g$. Each pentagon has a single edge entirely contained in $\partial Y$ called a $\partial$-edge. 
Let $G$ be the directed graph associated to $\M$, whose vertices are labeled by the elements of $\M$ and whose edges are $(r_i, r_j)\in \M\times \M$ where $g$ stretches $r_i$ over $r_j$. By a \textbf{cycle} or \textbf{path} in $G$ we mean a directed cycle or directed path, respectively. As in the case of closed surfaces, $G$ is a strongly connected directed graph, meaning that for any vertices $r_i,r_j$ of $G$, there exists a path from $r_i$ to $r_j$.

%Let $Y'=\intr(Y)\setminus \{\text{singularities of foliations of $g$}\}$, let $g'=g|_{Y'}$, and put $M'=M_{g'}$. The universal cover of $M'$ will be denoted by $\wt{M'}$. We define $\wt{Y'}$ and $\wh{Y'}$ as in Section \ref{canonical veering section}.
%

Let $\OO_\phi$ be the set of closed orbits of $\phi$, and let $L_G$ be the collection of cycles in $G$. Let $p_1\in \M$ be a pentagon with $\partial$-edge $e$. The image of $e$ under $g^i$ is a $\partial$-edge of some pentagon $p_i$, and there exists some finite $n$ for which $(p_1,\dots,p_n)$ is a cycle. Let $\partial L_G\subset L_G$ be the collection of all such cycles.

Let $\partial L_G$ be the collection of cycles in $L_G$ which traverse only vertices labeled by pentagons. A cycle in $\partial L_G$ is determined by any one of its pentagons. 

We now define a surjection $\omega\colon L_G\to \OO_\phi$. Let $\ell\in \partial L_G$, and pick one of its vertices labeled by a pentagon $p\in \M$. Let $e=p\cap \partial Y$. If we give $\partial Y$ the orientation induced by an outward-pointing vector field, we induce an orientation on $e$. Let $e_+$ be the positive endpoint of $e$ with respect to this orientation. Since the other edge of $p$ containing $e_+$ lies in a prong of the stable or unstable foliation of $g$, the orbit $o(e_+)$ of $e_+$ under $\phi$ is a $\partial$-singular orbit. Let $\omega(\ell)=o(e_+)$; this is well-defined, i.e. it does not depend on the initial choice of $p$. Next, any cycle $\ell\in L_G\setminus \partial L_G$ determines a unique closed orbit of $\phi$, just as in the case when $Y$ is a closed surface. We let $\omega(\ell)$ be this closed orbit. 

We say a path $(r_1,\dots, r_n)$ in $G$ is \textbf{simple} if $r_i\ne r_j$ for $1\le i,j\le n$, $i\ne j$. Similarly a cycle $(r_1,\dots, r_n)$, $r_1=r_n$ in $G$ is \textbf{simple} if $r_i\ne r_j$ for $2\le i,j\le n-1$, $i\ne j$. Let $S_G\subset L_G$ be the set of simple cycles in $G$. Since $\M$ is finite, $S_G$ is finite.

We now define a finite set $B\subset H_1(M)$. Let $B_1=\{[\omega(s)]\in H_1(M)\mid s\in S_G\}$. For every simple path $p=(r_1,\dots,r_n)\in G$ which is \textbf{almost closed}, i.e. starts and ends on vertices labeled by shapes in $\M$ meeting along an edge or vertex, let $o(p)$ be an orbit segment starting in $r_1$, passing sequentially through the $r_i$'s, and ending in $r_n$. Let $\epsilon(p)$ be a segment in $Y$ connecting the endpoints of $o(p)$ and supported in $r_1\cup r_2$. Let $B_2=\{[o(p)*\epsilon(p)]\mid \text{$p$ is an almost closed simple path in $G$}\}$, where $*$ is concatenation of paths. Let $B=B_1\cup B_2$. The salient feature of $B$ that will be used below is that it is finite, and hence bounded.

\begin{lemma}\label{it's a nice cone}
The cone $C_\phi$ of homology directions of $\phi$ is a finite-sided rational convex polyhedral cone.
\end{lemma}

\begin{proof}
First, we claim that $\C_\phi$ is generated by the set of homology classes of orbits of $\phi$, $\{[o]\mid o\in\OO_\phi\}$.

Let $\delta\in D_\phi$. By Lemma \ref{infinitet}, $\delta$ admits a closing sequence $(m_k,t_k)$ based at $m\in Y$ with $m_k\in Y$, $t_k\in \Z_+$ for all $k$, and $t_k\to\infty$. 
%If $m\in \partial Y$, then $m$ must lie on a $\partial$-singular orbit $o_\partial$ of $\phi$ on a boundary component $T$ of $M$. We can find a loop $\ell$ in $G$ such that $[\omega(\ell)]=[o_\partial]$: take one of the rectangles $r\in\M$ containing $m$ (there are two choices for $r$) and let $\ell$ be the loop in $G$ which starts at the vertex corresponding to $r$ and tranverses only vertices coming from rectangles touching $\partial Y$.
%Hence
%\[
%\delta=\pi([o_\partial])=\pi([\omega(\ell)]),
%\]
%so the claim is true for homology directions admitting closing sequences based on $\partial Y$. 
%
%We therefore assume $m\in \intr(M)$. 
Consider the curves $\gamma_k$, which for sufficiently large $k$ we can express as 
\[
\gamma_k=o_k*\epsilon_k,
\]
where $o_k$ is the curve $\phi_t(m_k)$, $0\le t\le t_k$ and $\epsilon_k$ is a short curve in $Y$ from $\phi_{t_k}(m_k)$ to $m_k$ supported in the union of at most 2 shapes of $\M$. We can lift $o_k$ to a path $r(o_k)=(r_1,\dots,r_{n(k)})$ in $G$. For sufficiently large $k$, $r(o_k)$ is not simple. Let $\ell(o_k)=(r_a,r_{a+1},\dots, r_{b-1},r_b)$ be the longest cycle subpath of $r(o_k)$. Then $(r_1,\dots,r_{a-1},r_a,r_{b+1},\dots,r_{n(k)})$ is simple and corresponds to some orbit segment $q_k$ with endpoints in rectangles which are either equal or intersect along an edge or vertex. Let $s_k$ be a segment connecting the endpoints of $q(k)$ supported in $r_1\cup r_{n(k)}$. We have
\[
[\gamma_k]=[q_k*s_k]+[\omega(\ell_k)],
\]
and $[q_k*s_k]\in B$. Letting $k\to\infty$, the intersection of $[\gamma_k]$ with $[Y]$ approaches infinity so $\{[\gamma_k]\}$ is unbounded. Since $B$ is bounded, we conclude
\[
\lim_{k\to\infty}\pi([\gamma_k])=\lim_{k\to\infty}\pi(\omega(\ell_k)),
\]
so $\delta$ is projectively approximated by homology classes of closed orbits. Since the homology class of each closed orbit lies in $\C_\phi$, we have that $\C_\phi$ is the smallest closed cone containing the homology classes of closed orbits of $\phi$ as claimed.

Next we will show that $\C_\phi$ is convex. Let $\lambda_1,\lambda_2$ be two closed orbits of $\phi$ that respectively pass through rectangles $r_1$ and $r_2$. Let $\ell_i$, $i=1,2$ be cycles in $G$ such that $\omega(\ell_i)=\lambda_i$. Let $\nu_{1,2}$ (resp. $\nu_{2,1}$) be a path in $G$ from $r_1$ to $r_2$ (resp. $r_2$ to $r_1$). Letting
\[
\gamma_n=\omega((\ell_1)^n * \nu_{1,2} * (\ell_2)^n * \nu_{2,1}),
\]
we have
\[
[\gamma_n]=n[\lambda_1]+n[\lambda_2]+[\omega(\nu_{1,2}* \nu_{2,1})].
\]
Therefore
\[
\pi([\lambda_1]+[\lambda_2])=\lim_{n\to\infty}\pi([\gamma_n])\in D_\phi,
\]
so $[\lambda_1]+[\lambda_2]\in \C_\phi$ and $\C_\phi$ is convex.

It remains to show that $\C_\phi$ is finite-sided and rational. To do this, it suffices to show that $\C_\phi$ is the convex cone generated by $S_\phi=\{[\omega(\ell)]\mid \ell\in S_G\}$. It is clear that $S_\phi\subset \C_\phi$. On the other hand, let $o\in \OO_\phi$. There exists some cycle $\ell$ in $G$ such that $\omega(\ell)=o$, and $\ell$ is a concatenation of simple cycles $\ell=s_1*\cdots*s_n$. By cutting and pasting, we see
\[
[o]=[\omega(s_1)]+\cdots+[\omega(s_n)].
\]
Hence $\OO_\phi$ is contained in the cone generated by $S_\phi$, so $\C_\phi$ is also.
\end{proof}

%
%In \cite{Fri82}, Fried observes that when a flow admits a cross section, its homology directions can be determined by considering only closing sequences $(m_k,t_k)$ based at points $m$ such that $m$ and $m_k$ lie in the cross section for all $k$, and $t_k\in \Z_+$ for all $k$ (assuming that the flow is parameterized so that the first return map is given by flowing for time 1).
%
%In our setting, we have a cross section $Y$ to $\phi$ with pseudo-Anosov first return map $g$. By the construction in \cite{FLP79} there is a Markov partition of $Y$, where the definition of Markov partition is altered slightly to accomodate $\partial Y$ (Markov ``rectangles" meeting $\partial Y$ are actually pentagons). There is a map from the collection of cycles in this directed graph to the closed orbits of $\phi$ and as in \cite{Fri79}, $\C_\phi$ is the smallest convex cone containing the homology classes of the images of cycles which do not contain the same vertex twice. There are finitely many such cycles, and all of them give integral points in $H_1(M)$.

Let $F$ be a circular flow on $M$. Following Fried, we define two sets in $H^1(M)$:
\begin{align*}
\C_\R(F)&=\{u\in H^1(M)\mid u(D_F)>0\},\text{ and}\\
\C_\Z(F)&=\{u\in\C_\R(F)\mid \text{$u$ is an integral point}\}.
\end{align*}

We can think of $\C_\R(F)$ as the set of linear functionals on $H_1(M)$ which are positive on $\pi^{-1}(D_F)$. Since this is an open condition, $\C_\R(F)$ is an open cone, and it is also clearly convex.

\begin{proposition}\label{disjoint or equal}
Let $F$ and $F'$ be two circular pseudo-Anosov flows on $M$. Then $\C_\R(F)$ and $\C_\R(F')$ are either disjoint or equal.
\end{proposition}

\begin{proof}
Suppose $\C_\R(F)\cap\C_\R(F')$ is nonempty. We will show $D_{F}=D_{F'}$ and hence $\C_\R(F)=\C_\R(F')$.

The intersection is open, so we can find a primitive class $u\in\C_\Z(F)\cap\C_\Z(F')$. By Theorem \ref{cross section theorem}, there are fibrations $f,f'\colon M\to S^1$ whose fibers are transverse to $F$ and $F'$ respectively, and are homologous. Let $Z$ and $Z'$ be fibers of $f$ and $f'$, respectively. By \cite{Thu86}, $Z'$ is isotopic to $Z$. By the isotopy extension theorem, the isotopy extends to an ambient isotopy of $M$. 

Let $F''$ be the image of $F'$ under this isotopy; $Z$ is a cross section of $F$ and $F''$. We reparametrize $F$ and $F''$ so that the first return maps $\rho,\rho''\colon Z\to Z$ of $F$ and $F''$ are given by flowing for time 1 along the respective flows.

The maps $\rho$ and $\rho''$ are both pseudo-Anosov representatives of the same isotopy class, so they are strictly conjugate. This means that there exists a map $g\colon Z\to Z$  which is isotopic to the identity such that
\[
\rho\circ h=h\circ \rho''.
\]
This isotopy extends to an ambient isotopy of $M$. Let $F'''$ denote the image of $F''$ under this ambient isotopy. By construction the first return map of $F'''$ on $Z$ is $\rho$.

%\marginpar{\texttt{come back to this}}
%We now perform an ambient isotopy of $M$ which fixes $Z$ and carries the foliation $\{F'''_t(Z)\mid 0\le t<1\}$ to the foliation $\{F_t(Z)\mid 0\le t<1\}$. Denote the image of $F'''$ under this isotopy by $F^4$. 

Now we need a lemma.

\begin{lemma}\label{crisis lemma}
Let $p\in \partial Z$ be the boundary point of a leaf $\ell_p$ of the stable or unstable foliation of $\rho$. Then $F_t(p)$ and $F'''_t(p)$, $0\le t\le 1$ are homotopic in $\partial M$ rel endpoints.
\end{lemma}

\begin{proof}[Proof of Lemma \ref{crisis lemma}]
We first cut $M$ open along $Z$. The result is a manifold with boundary that we can identify with $Z\times[0,1]$, such that $F$ is identified with the vertical flow. Let $\pi_Z\colon Z\times[0,1]\to Z\times \{0\}$ be the projection. We identify $Z\times\{0\}$ with $Z$, and endow $Z$ with the singular flat metric it inherits from the stable and unstable foliations preserved by $\rho$.

Consider the homotopy $g_t\colon Z\times [0,1]\to Z$ given by $g_t(z)=\pi_Z(F'''_t(z))$. We claim that $g_t(p)$, $0\le t\le 1$ is not an essential loop in $\partial Z$.

Let $\tilde Z$ be the universal cover of $Z$, and let $\tilde g_t$ be the unique homotopy of $\id_{\tilde Z}$ that covers $g_t$. We see that $g_t$ preserves each component of the union of lines in $\partial\tilde Z$ covering $\partial Z$. Hence it fixes the ends of $\tilde Z$. 

Let $\tilde\ell_p$ be a lift of $\ell_p$ to $\tilde Z$. It is a geodesic ray $[0,\infty)\to \tilde Z$ with one endpoint on a lift $\wt{\partial_\ell}$ of a component of $\partial Z$ and the other end exiting an end of $\tilde Z$.

Suppose $g_t(p)$, $0\le t\le 1$ is essential in $\partial Z$. Then $\tilde g_1$ carries $\tilde \ell_p$ to a separate lift $\wt{\ell_p}'$ of $\ell_p$ with its endpoint on $\wt{\partial_\ell}$. Since $\wt{g_1}$ fixes the ends of $\wt Z$, $\ell_p$ and $\wt{\ell_p}$ exit the same end. But this is impossible, because the lifts of stable and unstable foliations to $\wt Z$ are such that no two leaves fellow travel.

It follows that $F'''_t(p)$, $0\le t\le 1$ can be homotoped rel endpoints to a vertical arc in $\partial Z\times [0,1]$, proving the claim.
\end{proof}

With our Lemma in hand we can finish proving Proposition \ref{disjoint or equal}. We define a map conjugating $F$ and $F'''$, which we will show is isotopic to the identity. For $z\in Z$, $t\in[0,1]$ let 
\[
C(F_t(z))=F'''_t(z).
\]
As the first return maps of $F$ and $F'''$ to $Z$ are both equal to $\rho$, $C$ is well-defined.

Fix a basepoint $z_o\in \partial Z$ lying in a $\partial$-singular orbit $o$ of $F$ %and let $\zeta_1,\dots,\zeta_n$ be curves in $Z$ generating $\pi_1(Z,z_0)$. 
and let $\zeta_o$ be a curve which starts at $z_o$, travels along $o$ for time 1, and returns to $z_o$ via a path in $Z$. If $G$ is a set of generators of $\pi_1(Z,z_o)$ then $G\cup\{\zeta\}$ generates $\pi_1(M,z_o)$. Since $C$ restricted to $Z$ is the identity, $C_*\colon \pi_1(M,z_o)\to \pi_1(M,z_o)$ fixes each element of $G$. Since $C_*$ also fixes $[\zeta_o]$ by Lemma \ref{crisis lemma}, we see $C_*$ is the identity map. Since $M$ is a $K(\pi_1(M),1)$ space, $C$ must be homotopic to the identity. In fact, $C$ is isotopic to the identity by a theorem of Waldhausen \cite{Wal68} which states that any homeomorphism of a compact, irreducible, boundary irreductible, Haken 3-manifold which is homotopic to the identity is isotopic to the identity.

Conjugating a flow by a homeomorphism isotopic to the identity does not change its set of homology directions. Hence $D_F=D_{F'}$ so $C_\R(F)=\C_\R(F')$ as desired.
\end{proof}

Let $\sigma_{\LD}$ denote the image of $\sigma$ under the Lefschetz duality isomorphism $H_2(M,\partial M)\cong H^1(M)$.

\begin{proposition}\label{flows determine faces}
$\C_\R(\phi)=\intr(\cone(\sigma_{\LD}))$.
\end{proposition}

\begin{proof}
Let $\alpha$ be Lefschetz dual to a class in $\C_\Z(\phi)$. By Theorem \ref{cross section theorem}, $\alpha$ is represented by a cross section $\Sigma$ to $\phi$. By \cite{Thu86}, $\alpha$ lies interior to the cone over \emph{some} top-dimensional face of $B_x(M)$. We show that this face is in fact $\sigma$.

As a leaf of a taut foliation, $\Sigma$ is taut. 
Since $\Sigma$ is a cross section to $\phi$, the tangent plane field $T\Sigma$ of the fibration $\Sigma\hookrightarrow M\to S^1$ is homotopic to $\xi_\phi$ (recall from Section \ref{norm background} that $\xi_\phi$ is the quotient of $T_M$ by $T_\phi$, the tangent line bundle to the 1-dimensional foliation by flowlines of $\phi$). The same is true for $TY$, so the relative Euler classes of the two plane fields are equal. Let $e_Y$ denote this Euler class. We have
\[
x(\alpha)=-\chi(\Sigma)=-e_Y(\alpha),
\]
so $\alpha$ lies in the portion of $H_2(M,\partial M)$ where $x$ agrees with $-e_Y$. By the discussion in Section \ref{norm background}, this is $\cone(\sigma)$.

It follows that $\C_\Z(\phi)\in \intr(\cone(\sigma_{\LD}))$, so every rational point in $\C_\R(\phi)$ lies in $\intr(\cone(\sigma_{\LD}))$. Since $\C_\R(\phi)$ and $\intr(\cone(\sigma_{\LD}))$ are both open, $\C_\R(\phi)\subset\intr(\cone(\sigma_{\LD}))$.

Now suppose that $\C_\R(\phi)\subsetneq\intr(\cone(\sigma_{\LD}))$. 
We have $\C_\R(\phi)=\intr(\C_\phi^\vee)$. By Lemma \ref{it's a nice cone}, $\C_\phi$ is a rational convex polyhedral cone, so $\C_\R(\phi)$ is the interior of a rational convex polyhedral cone. Hence there is an integral cohomology class $v\in \intr(\cone(\sigma_{\LD}))\cap \partial\C_\R(\phi)$. 

The Lefschetz dual of $v$ is represented by a cross section to another circular pseudo-Anosov flow $\phi'$. We must have $\C_\R(\phi)\cap \C_\R(\phi')\ne \varnothing$, but the cones cannot be equal because $v\notin \C_\R(\phi)$. This contradicts Lemma \ref{disjoint or equal}, so we conclude that $\C_\R(\phi)=\intr(\cone(\sigma_{\LD}))$.
\end{proof}

We remark that the proof of the inclusion $\C_\R(\phi)\subset\intr(\cone(\sigma_{\LD}))$ did not require $\phi$ to be pseudo-Anosov, so the corresponding statement is still true if we replace $M$ by any compact 3-manifold and $\phi$ by any circular flow.

We conclude this section by observing that we have proven Theorem \ref{generalized Fried}.

\begin{theorem}\label{generalized Fried}
Let $\phi$ be a circular pseudo-Anosov flow on a compact 3-manifold $M$ with cross section $Y$. Let $\sigma$ be the fibered face of $B_x(M)$ such that $[Y]\in\cone(\sigma)$. Let $\alpha\in H_2(M,\partial M)$ be an integral class. The following are equivalent:
\begin{enumerate}
\item $\alpha$ lies in $\intr(\cone(\sigma))$
\item the Lefschetz dual of $\alpha$ is positive on the homology directions of $\phi$
\item $\alpha$ is represented by a cross section to $\phi$.
\end{enumerate}
Moreover, $\phi$ is the unique circular pseudo-Anosov flow admitting cross sections representing classes in $\cone(\sigma)$ up to reparameterization and conjugation by homeomorphisms of $M$ isotopic to the identity.
\end{theorem}

\begin{proof}[Proof of Theorem \ref{generalized Fried}]
$1\Leftrightarrow2$: This is a restatement of Proposition \ref{flows determine faces}.

$2\Leftrightarrow 3$: This is a restatement of Theorem \ref{cross section theorem}.

The truth of the last claim can be seen from the proof of Proposition \ref{disjoint or equal}. Recall that we showed that if $F^1$, $F^2$ are circular pseudo-Anosov flows admitting homologous cross sections then they are conjugate by a homeomorphism of $M$ isotopic to the identity.
\end{proof}

\section{Face-spanning taut homology branched surfaces in manifolds with boundary}\label{TBSexpanded}

Let $M$ be a 3-manifold such that $x$ is a norm on $H_2(M,\partial M)$, and let $B$ be a taut branched surface in $M$. The cone of homology classes carried by $B$ is contained in $\cone(F)$ for some face $F$ of $B_x(M)$ (this is because one can see, via cutting and pasting surfaces carried by $B$, that $x$ is linear on the cone of carried classes). If this cone of carried classes is \emph{equal} to $\cone(F)$ we say $B$ \textbf{spans} $F$. In \cite{Oer86}, Ulrich Oertel asked when a face of the Thurston norm ball is spanned by a single taut homology branched surface. Recall that \textbf{taut} means every surface carried by $B$ is taut, and that a \textbf{homology branched surface} has a closed oriented transversal through every point.

In \cite{Lan18} we gave a sufficient criterion for a fibered face of a closed hyperbolic 3-manifold to admit a spanning taut homology branched surface via a construction using veering triangulations. In this appendix we describe why that criterion is also sufficient in the broader setting of this paper, i.e. when the compact hyperbolic 3-manifold in question possibly has boundary.

The general result is the following.

\begin{theorem}\label{strongerTBS}
Let $\sigma$ be a fibered face of a compact hyperbolic 3-manifold, and let $\phi$ be the suspension flow of $\sigma$. If each singular orbit of $\phi$ witnesses at most 2 ladderpole boundary classes of $\sigma$ then there exists a taut branched surface $B_\sigma$ spanning $\sigma$.
\end{theorem}

A \textbf{ladderpole vertex class} is a primitive integral class $\alpha$ lying in a 1-dimensional face of $\cone(\sigma)$ such that $\mr\alpha$ is represented by a surface $\mr A$ carried by $B_{\mr\tau}$ and for some $U_i$, $\partial\mr A\cap U_i$ is a collection of ladderpole curves. Note that here we make no requirements on the boundary components of $\mr A$ which lie in $V$.

The technical lemma that allows us to prove Theorem \ref{strongerTBS} is the following, which was proven in \cite{Lan18} only for closed hyperbolic 3-manifolds.

\begin{lemma}\label{TBSapplemma}
Let $\sigma,\phi$ be as above. Let $\alpha\in\cone(\sigma)$ be an integral class. Then
\[
x(\alpha)=x(\mr\alpha)-i(\alpha,c),
\]
where $c$ is the union of the singular orbits of $\phi$.
\end{lemma}

\begin{proof}
Let $\mr A$ be a surface carried by $B_{\mr\tau}$ and representing $\mr\alpha$. By our proof of Theorem \ref{myTST}, there exists a surface $A$ which is almost transverse to $\phi$ and represents $\alpha$ with $\chi_-(A)=\chi_-(\mr A)-i(\alpha,c)$. Since $A$ is almost transverse to $\phi$, $A$ is taut. Since $B_{\mr\tau}$ is a taut branched surface, $\mr A$ is taut. Therefore $x(\alpha)=x(\mr\alpha)-i(\alpha,c)$.
\end{proof}

The proof above represents a significant shortening of the proof of the corresponding Lemma in \cite{Lan18}. The ingredients that make this possible are (a) we now know the Transverse Surface Theorem holds when our manifold has boundary and (b) we can assume our almost transverse surface representative of $\alpha$ lies in a neighborhood of $B_{\mr\tau}$ away from the singular orbits, and is simple in a neighborhood of the singular orbits.

With Lemma \ref{TBSapplemma} proven, the proof of Theorem \ref{strongerTBS} proceeds exactly as in \cite{Lan18}.

We once again observe that the condition on ladderpole vertex classes is satisfied when $\dimension(H_2(M,\partial M))\le3$, so we have the following corollary.

\begin{corollary}
Let $\sigma$ be a fibered face of a compact hyperbolic 3-manifold $M$ with $\dimension(H_2(M,\partial M))\le3$. Any fibered face of $B_x(M)$ is spanned by a taut homology branched surface.
\end{corollary}

Finally, as a special case of the above we observe that the result holds for exteriors of links with $\le3$ components. 

\begin{corollary}\label{link corollary}
Let $L$ be a fibered hyperbolic link with at most 3 components. Let $M_L$ be the exterior of $L$ in $S^3$. Any fibered face of $B_x(M_L)$ is spanned by a taut homology branched surface.
\end{corollary}

\end{appendices}

%%%%%%%%%%%%%%%%%%%%%%%%%%%%%%%%%%%%%%%%%%%%%%%%%%%%%%%%%%%%%%%%%%%%%%%%%%

\bibliographystyle{alpha}
\bibliography{bibliography}
\label{sec:biblio}
\bigskip

\textsc{Yale University}

\textit{Email address}: \texttt{michael.landry@yale.edu}

\end{document}